\documentclass[10pt]{amsart}
\allowdisplaybreaks

\usepackage{amsfonts}
\usepackage{amsmath}
\usepackage{amssymb}
\usepackage{amsthm}
\usepackage{graphicx} \usepackage{enumerate} \usepackage{multicol}
\usepackage{mathrsfs} \usepackage[all,cmtip]{xy}
\usepackage{enumerate}

\DeclareMathAlphabet{\mathpzc}{OT1}{pzc}{m}{it}

\newcounter{dummy} \numberwithin{dummy}{section}
\newtheorem{theorem}[dummy]{Theorem}
\newtheorem{corollary}[dummy]{Corollary}
\newtheorem{lemma}[dummy]{Lemma}
\newtheorem{definition}[dummy]{Definition}
\newtheorem{proposition}[dummy]{Proposition}
\theoremstyle{remark}
\newtheorem{remark}[dummy]{Remark}
\newtheorem{example}[dummy]{Example}

\DeclareMathAlphabet{\mathcal}{OMS}{cmsy}{m}{n}

\DeclareFontFamily{U}{mathx}{\hyphenchar\font45}
\DeclareFontShape{U}{mathx}{m}{n}{
      <5> <6> <7> <8> <9> <10>
      <10.95> <12> <14.4> <17.28> <20.74> <24.88>
      mathx10
      }{}
\DeclareSymbolFont{mathx}{U}{mathx}{m}{n}
\DeclareFontSubstitution{U}{mathx}{m}{n}
\DeclareMathAccent{\widecheck}{0}{mathx}{"71}
\DeclareMathAccent{\wideparen}{0}{mathx}{"75}

\DeclareMathOperator{\comp}{\mathbb C}
\DeclareMathOperator{\real}{\mathbb R}

\newcommand{\calE}{\mathcal E}
\newcommand{\calF}{\mathcal F}
\newcommand{\calH}{\mathcal H}
\newcommand{\calL}{\mathcal L}
\newcommand{\calR}{\mathcal R}
\newcommand{\calV}{\mathcal V}

\newcommand{\frakg}{\mathfrak g}
\newcommand{\frakh}{\mathfrak h}
\newcommand{\frakk}{\mathfrak k}

\newcommand{\scrF}{\mathscr F}
\newcommand{\scrM}{\mathscr M}
\newcommand{\scrN}{\mathscr N}

\DeclareMathOperator{\Ann}{Ann}

\DeclareMathOperator{\Sym}{Sym}
\DeclareMathOperator{\End}{End}

\DeclareMathOperator{\inc}{inc}
\DeclareMathOperator{\pr}{pr}
\DeclareMathOperator{\rank}{rank}
\DeclareMathOperator{\spn}{span}
\DeclareMathOperator{\tr}{tr}

\newcommand{\prob}{\mathbb P}
\newcommand{\expect}{\mathbb E}

\DeclareMathOperator{\Ad}{Ad}
\DeclareMathOperator{\ad}{ad}
\DeclareMathOperator{\Oct}{O}

\DeclareMathOperator{\tensorg}{\mathbf g}
\DeclareMathOperator{\dv}{div}

\DeclareMathOperator{\grad}{grad}
\DeclareMathOperator{\Ric}{Ric}
\DeclareMathOperator{\vol}{vol}
\DeclareMathOperator{\dvol}{\text{$d$}vol}
\DeclareMathOperator{\II}{II}

\newcommand{\tensorh}{\mathbf h}
\newcommand{\rnabla}{\mathring{\nabla}}
\newcommand{\znabla}{\,\mathring{\phantom{\nabla}\!\!\!\!\!\!}\mathring{\!\!\nabla}}
\newcommand{\shh}{\sharp^{{\mathbf h}^*}}
\newcommand{\srL}{\Delta_{\mathbf h}}
\newcommand{\roughL}{\Delta_{\mathbf h}^\prime}

\newcommand{\tensorv}{\mathbf v}

\newcommand{\shv}{\sharp^{{\mathbf v}^*}}

\newcommand{\tensors}{\mathbf s}

\newcommand{\tensorq}{\mathbf q}
\newcommand{\metricd}{\mathsf d}
\newcommand{\ve}{\varepsilon}
\newcommand{\sfGamma}{\mathsf \Gamma}

\newcommand{\RicH}{\Ric_{\mathcal{H}}}

\newcommand{\RicHV}{\Ric_{\mathcal{HV}}}
\newcommand{\rRicH}{\rho_{\mathcal{H}}}
\newcommand{\MRicHV}{\mathscr{M}_{\mathcal{HV}}}
\newcommand{\McalR}{\mathscr{M}_{\mathcal{R}}}
\newcommand{\mcalR}{m_{\mathcal{R}}}
\newcommand{\MII}{\mathscr{M}_{\rnabla \tensorv^*}}
\newcommand{\rLv}{\rho_{\roughL \! \tensorv^*}}

\newcommand\newbullet{{\kern.8pt\displaystyle\centerdot\kern.8pt}}
\newcommand\newbulletsub{{\centerdot\kern.8pt}}

\numberwithin{equation}{section}

\title[Curvature-dimension inequalities on sR-manifolds, Part~I]{Curvature-dimension inequalities on sub-Riemannian manifolds obtained from Riemannian foliations, Part~I}
\author[E. Grong, A. Thalmaier]{Erlend Grong \\ Anton Thalmaier}

\address{Mathematics Research Unit, University of Luxembourg, 6 rue Richard Coudenhove-Kalergi, L-1359 Luxembourg}
\email{erlend.grong@uni.lu}

\address{Mathematics Research Unit, University of Luxembourg, 6 rue Richard Coudenhove-Kalergi, L-1359 Luxembourg}
\email{anton.thalmaier@uni.lu}

\subjclass[2010]{58J35 (53C17,58J99)}

\keywords{Hypoelliptic operators; Sub-Riemannian geometry; Curvature-dimension inequality; Spectral gap; Riemannian foliations}

\begin{document}

\begin{abstract}
We give a generalized curvature-dimension inequality connecting the
  geometry of sub-Riemannian manifolds with the properties of its
  sub-Laplacian. This inequality is valid on a large class of
  sub-Riemannian manifolds obtained from Riemannian foliations. We
  give a geometric interpretation of the invariants involved in the
  inequality. Using this inequality, we obtain a lower bound for the
  eigenvalues of the sub-Laplacian. This inequality also lays the foundation for
  proving several powerful results in Part~II. 
\end{abstract}

\maketitle

\section{Introduction} \label{sec:Introduction} On a given connected
manifold $M$, there is a well established relation between elliptic
second order differential operators on $M$ and Riemannian geometries
on the same space. More precisely, for any smooth elliptic operator
$L$ on $M$ without constant term, there exist a unique Riemannian
metric $\tensorg$ on $M$ such that for any pair of smooth functions
$f,g \in C^\infty(M)$,
\begin{equation}\label{GammaLG} \sfGamma(f,g): = \frac{1}{2} \left( L
    (fg) - fLg - g L f \right) = \langle\grad f,\grad
  g\rangle_{\tensorg}. \end{equation}
Conversely, from any Riemannian metric $\tensorg$, we can construct a second order operator satisfying \eqref{GammaLG} in the form of the Laplacian $\Delta$. In addition, the properties of $L$ and $\tensorg$ are intimately connected. Consider the case when $M$ is complete with respect to $\tensorg$ and write~$P_t = e^{t/2\Delta}$ for the heat semigroup corresponding to~$\frac{1}{2} \Delta$. Then the following statements are equivalent for any $\rho \in \real$.
\begin{itemize}
\item[(a)] For any $f \in C_c^\infty(M)$, $\rho \|\grad
  f\|^2_{\tensorg} \leq \Ric_{\tensorg}(\grad f, \grad
  f)$. \vspace{0,2cm}
\item[(b)] For any $f \in C_c^\infty(M)$, $\|\grad P_t
  f\|^2_{\tensorg} \leq e^{-\rho t} P_t \| \grad f\|^2_{\tensorg}.$
  \vspace{0,2cm}
\item[(c)] For any $f \in C_c^\infty(M)$, $\frac{1- e^{- \rho
      t}}{\rho} \| \grad P_t f\|_{\tensorg}^2 \leq P_t f^2 - (P_t
  f)^2.$ \vspace{0,2cm}
\end{itemize}
Here, $C_c^\infty(M)$ is the space of smooth functions on~$M$ with
compact support and $\Ric_{\tensorg}$ is the Ricci curvature tensor
of~$\tensorg$. This equivalence gives us a way of understanding Ricci
curvature in terms of growth of the gradient of a solution to the heat
equation. With appropriate modifications of the Ricci curvature,
the same statement holds for a general elliptic operator $L$
satisfying \eqref{GammaLG}, giving us a geometric tool to study the heat flow of
elliptic operators. See e.g. \cite{Wan04} and references therein for
the full statement.

Let us now consider the case when $L$ is not elliptic, but is rather
given locally in a form
\begin{equation} \label{Ltype} L = \sum_{i=1}^n A_i^2 + \text{first
    order terms}.\end{equation} Here, $A_1, \dots, A_n$ are linearly
independent vector fields and $n$ is strictly less than the dimension of
$M$. In this case, there is still a geometry we can associate with $L$
by considering the subbundle spanned by $A_1, \dots, A_n$ furnished
with a metric tensor that makes these vector fields orthogonal. Such a
geometry is called sub-Riemannian geometry. A sub-Riemannian manifold
is a connected manifold $M$ with a positive definite metric
tensor~$\tensorh$ defined only on a subbundle~$\calH$ of the tangent
bundle $TM$. If we assume that sections of~$\calH$ and their iterated
Lie brackets span the entire tangent bundle, we obtain a metric
space~$(M, \metricd_{cc})$, where the distance between two points is
defined by taking the infimum of the lengths of all curves tangent
to~$\calH$ that connect these
points.

In recent years, understanding how to define curvature in
sub-Riemannian geometry has become an important question. One approach
has been to introduce curvature by studying invariants of the flow of
normal geodesics associated to the sub-Riemannian structure,
see~e.g.~\cite{ZeLi07,LiZe11,ABR13,BaRi14}. The other approach
explores the interaction of the sub-Riemannian gradient and second
order operators of the form~\eqref{Ltype}. We will follow the latter
approach.

The equivalence of statements (a), (b) and (c) mentioned previously is
rooted in 
the curvature-dimension inequality for Riemannian manifolds. In the
notation of Bakry and \'Emery~\cite{BaEm85} this inequality is written
as
$$\sfGamma_2(f) \geq \frac{1}{n} (L f)^2 + \rho \sfGamma(f), \qquad f \in C^\infty(M).$$
Here, $n = \dim M$, $L = \Delta$, $\rho$ is a lower bound for the Ricci
curvature and, for any pair of functions $f,g \in C^\infty(M)$,
\begin{align} \label{Sec1Gamma}
  \sfGamma(f,g) &=  \frac{1}{2} \left(L (fg) - f L g - g L f \right) = \langle\grad f,\grad g\rangle_{\tensorg}, &\sfGamma(f) = \sfGamma(f,f), \\
  \label{Sec1Gamma2}
  \sfGamma_2(f, g) &= \frac{1}{2} \left(L \sfGamma(f,g) - \sfGamma(f
    ,L g) - \sfGamma(L f, g)\right) , &\sfGamma_2(f) =
  \sfGamma_2(f,f).
\end{align}
Even in simple cases, this inequality fails in the sub-Riemannian
setting. The following generalization has been suggested by F.~Baudoin
and N.~Garofalo in~\cite{BaGa12}.

Let~$\tensorh$ be a sub-Riemannian metric defined on a
subbundle~$\calH$ of~$TM$. Let~$L$ be any second order operator as in
\eqref{Ltype}, i.e. locally given as a sum of squares of an orthonormal basis of $\calH$ plus a first order term.
We remark that, unlike the Laplace operator on a
Riemannian manifold, the operator~$L$ is not uniquely determined
by~$\tensorh$ unless we add some additional structure such as a
chosen preferred volume form on~$M$. Define $\sfGamma$ and
$\sfGamma_2$ as in \eqref{Sec1Gamma} and \eqref{Sec1Gamma2}. For any
positive semi-definite section $\tensorv^*$ of $\Sym^2 TM$, define
$\sfGamma^{\tensorv^*}(f,g) = \tensorv^*(df,dg)$ and
$\sfGamma^{\tensorv^*}(f) = \sfGamma^{\tensorv^*}(f,f).$ Let
$\sfGamma^{\tensorv^*}_2(f)$ be defined analogous to 
$\sfGamma_2(f)$ in \eqref{Sec1Gamma2}. Then $L$ is said to satisfy a
generalized curvature-dimension inequality if we can choose
$\tensorv^*$ such that for every $\ell >0,$
\begin{equation} \label{BaGaCD} \sfGamma_2(f) + \ell\,
  \sfGamma^{\tensorv^*}_2(f) \geq \frac{1}{n} (Lf)^2 + \left(\rho_1 -
    \ell^{-1} \right) \sfGamma(f) + \rho_2
  \sfGamma^{\tensorv^*}(f),\end{equation} for some $1\leq n \leq
\infty$, $\rho_1 \in \real$ and $\rho_2 > 0$. Using this inequality, the
authors were able to prove several results, such as gradient bounds,
Li-Yau type inequality and a sub-Riemannian version of the
Bonnet-Myers theorem. See also further results based of the same
formalism in~\cite{BaBo12,BBG14,BaWa14,Hla13}.

So far, the examples of sub-Riemannian manifolds satisfying
\eqref{BaGaCD} all have a complement to $\calH$ spanned by the
sub-Riemannian analogue of Killing vector fields. We want to show that
a further generalization of \eqref{BaGaCD} holds for a larger
class of sub-Riemannian manifolds. We also want to give an
interpretation for the constants involved in the curvature-dimension
inequality. Results following from this inequality are important for
sub-Riemannian geometry, understanding solutions of the heat equation
of operators $L$ described locally as in \eqref{Ltype}, along with the
stochastic processes which have these operators as their infinitesimal
generators.

In order to motivate our approach, let us first consider the following
example. Let~$\pi: M \to B$ be a submersion between two connected
manifolds and let $\widecheck{\tensorg}$ be a Riemannian metric
on~$B$. Let~$\calV = \ker \pi_*$ be the vertical bundle and let
$\calH$ be an Ehresmann connection on $\pi$, that is, a subbundle such
that $TM = \calH \oplus \calV$. We then define a sub-Riemannian metric
$\tensorh$ on~$M$ by pulling back the Riemannian metric,
i.e.~$\tensorh = \pi^* \widecheck{\tensorg}|_{\calH}$. In this case,
we have two notions of curvature that could be expected to play a role
for the inequality of type \eqref{BaGaCD}, namely the Ricci curvature
of $B$ and the curvature of the Ehresmann connection $\calH$ (see
Section~\ref{sec:GeometryRnabla} for definition). After all, our
sub-Riemannian structure is uniquely determined by a metric on $B$ and
a choice of Ehresmann connection. For this reason, examples of this
type should be helpful in providing a geometric understanding of
curvature in sub-Riemannian geometry. However, we have to deal with
the following two challenges.
\begin{enumerate}[(i)]
\item Even though the sub-Riemannian geometry on $M$ can be considered
  as ``lifted'' from $B$ the same cannot be said for our operator
  $L$. That is, if $L$ is of the type \eqref{Ltype},
  $\widecheck{\Delta}$ is the Laplacian on $B$ and $f \in C^\infty(B)$
  is a smooth function on $B$, then $L(f \circ \pi)$ does not coincide
  with $(\widecheck{\Delta} f) \circ \pi$ in general.
\item The same sub-Riemannian structure on $M$ can sometimes be
  considered as lifted from two different Riemannian manifolds (see
  Section~\ref{sec:2Vcomp} for an example).
\end{enumerate}
Our approach to overcome these challenges will be the following.

In Section~\ref{sec:Definitions}, we introduce the basics of
sub-Riemannian manifolds and sub-Laplacians. We overcome the
challenges of (i) and (ii) by introducing a unique way of choosing $L$
relative to a complemental subbundle of $\calH$ rather than a volume
form. This will have exactly the desired ``lifting property''. We
discuss the diffusions of such operators in terms of stochastic
development. In Section~\ref{sec:CurvatureDimension} we introduce a
preferable choice of complement, which we call metric-preserving
complements. Roughly speaking, such complements correspond to
Riemannian foliations. While such a complement may not always exist,
all sub-Riemannian manifolds discussed so far have such a complement.
We give geometric conditions for when a sub-Riemannian manifold with a
metric-preserving complement satisfies a generalization of the
curvature-dimension inequality~\eqref{BaGaCD}. From this inequality,
we immediately get a result on the spectral gap of $L$ found in
Section~\ref{sec:Examples}. In the same section, we also apply our results to some examples. 

In Part~II we will look at further consequences of the
curvature-dimension inequality in Theorem~\ref{th:CD}. A short summary
of these results are given Section~\ref{sec:PartII}

In parallel with the development of our paper, a generalized curvature-dimension appeared in \cite{BKW14} for the case of sub-Riemannian manifolds obtained from Riemannian foliations with totally geodesic leaves that are of Yang-Mills type. See Remark~\ref{re:Baudoin} for details.

\subsection{Notations and conventions}
Unless otherwise stated, all manifolds are connected. If $\calE \to M$
is a vector bundle over a manifold $M$, its space of smooth sections
is written $\Gamma(\calE)$. If $s \in \Gamma(\calE)$ is a section, we generally
prefer to write $s|_x$ rather than $s(x)$ for its value in $x \in
M$. By a metric tensor $\tensors$ on $\calE$, we mean a smooth section
of $\Sym^2 \calE^*$ which is positive definite or at least positive
semi-definite. For every such metric tensor, we write
$\|e\|_{\tensors} = \sqrt{\tensors(e,e)}$ for any $e \in \calE$ even
if $\tensors$ is only positive semi-definite. All metric tensors are
denoted by bold, lower case Latin letters (e.g. $\tensorh,
\tensorg,\dots$). We will only use the term Riemannian metric for a
positive definite metric tensor on the tangent bundle. If $\tensorg$
is a Riemannian metric, we will use $\tensorg^*$, $\wedge^k
\tensorg^*, \dots$ for the metric tensors induced on $T^*M$,
$\bigwedge^k T^*M, \dots$.

If $\alpha$ is a form on a manifold $M$, its contraction or interior
product by a vector field $A$ will be denoted by either
$\iota_A\alpha$ or $\alpha(A, \newbullet).$ We use $\calL_A$ for the
Lie derivative with respect to $A$.  If $M$ is furnished with a Riemannian metric $\tensorg$, any bilinear tensor $\tensors: TM \otimes TM \to \mathbb{R}$ can be identified with an endomorphism of $TM$ using $\tensorg$. We use the notation $\tr \, \tensors(\times, \times)$ for the trace of this corresponding endomorphism, with the metric being implicit. If $\calH$ is a subbundle of $TM$, we will also use the notation $\tr_{\calH} \tensors(\times, \times ) := \tr \tensors(\pr_{\calH} \times, \pr_{\calH} \times)$, where $\pr_{\calH}$ is the orthogonal projection to $\calH$.


\section{Sub-Riemannian manifolds and
  sub-Laplacians} \label{sec:Definitions}

\subsection{Definition of a sub-Riemannian
  manifolds} \label{sec:srmanifolds} {\it A sub-Riemannian manifold}
can be considered as a triple $(M, \calH, \tensorh)$ where $M$ is a
connected manifold, $\calH$ is a subbundle of $TM$ and $\tensorh$ is a
positive definite metric tensor defined only on the subbundle
$\calH$. The pair $(\calH, \tensorh)$ is called {\it a sub-Riemannian
  structure} on $M$. Any sub-Riemannian structure induces a vector
bundle morphism
$$\shh\colon T^*M \to TM,$$
determined by the properties $\shh(T^*M) = \calH$ and $p(v) =
\tensorh(v,\shh p)$ for any $p \in T^*M$ and $v \in \calH$. The kernel
of $\shh$ is the subbundle $\Ann(\calH) \subseteq T^*M$ of all
elements of $T^*M$ that vanish on $\calH$. We can define a co-metric
$\tensorh^*$ on $T^*M$ by
$$\tensorh^*(p_1, p_2) = p_1(\shh p_2), \qquad p_1, p_2 \in T_x^*M,\ x \in M,$$
which obviously degenerates along $\Ann(\calH)$. A sub-Riemannian
manifold can therefore equivalently be considered as a pair $(M,
\tensorh^*)$ where $M$ is a connected manifold and $\tensorh^*$ a
co-metric degenerating along a subbundle of $T^*M$. We will use both
of these point of views throughout our paper, referring to the
sub-Riemannian structure $(\calH, \tensorh)$ and $\tensorh^*$
interchangeably.

We will call any absolutely continuous curve $\gamma$ in $M$ {\it
  horizontal} if $\dot \gamma(t) \in \calH_{\gamma(t)}$ for almost all $t$.  We
define {\it the Carnot-Carath\'eodory distance} $\metricd_{cc}$ on $M$
as
$$\metricd_{cc}(x, y)  = \inf_\gamma \left\{\int_0^1 \tensorh(\dot{\gamma}, \dot{\gamma})^{1/2} \, dt \, : \, \gamma(0) = x ,\ \gamma(1) = y,\ \gamma \text{ horizontal } \right\}.$$
This distance is finite for any pair of points if they can be
connected by at least one horizontal curve. A sufficient condition for
the latter to hold is that $\calH$ is {\it
  bracket-generating}~\cite{Cho39,Ras38}. A subbundle~$\calH$ is
called bracket-generating if its sections and their iterated brackets
span the entire tangent bundle. The same property also guarantees that
the metric topology induced by~$\metricd_{cc}$ coincides with the
manifold topology on $M$, however, the Hausdorff dimension of
$\metricd_{cc}$ will in general be greater than the topological
dimension (see e.g.~\cite[Th. 2.3, Th. 2.17 ]{Mon02}).

From now on, the rank of $\calH$ is $n \geq 2$ while the manifold $M$
is assumed to have dimension $n + \nu.$ We will refer to $\calH$ as
{\it the horizontal bundle} and its vectors and sections as {\it
  horizontal} and we refer to both $(\calH, \tensorh)$ and
$\tensorh^*$ as a sub-Riemannian structure on $M$.


\subsection{Second order operators associated to $\tensorh^*$}
For any manifold $M$, let $T^2M$ denote the second order tangent
bundle. Sections $L \in \Gamma(T^2M)$ of this bundle can locally be
expressed as
\begin{equation} \label{localExpression} L = \sum_{i,j=1}^{n+\nu} L_{ij} \frac{\partial}{\partial x_i \partial x_j}  + \sum_{j=1}^{n+\nu} L_j \frac{\partial}{\partial x_j} \, ,\end{equation}
relative to some local coordinate system $(x_1, \dots, x_{n+\nu})$ and
functions $L_{ij} = L_{ji}$ and~$L_j$.  We can consider $TM$ as a
subbundle of $T^2M$. Denote its inclusion by $\inc$. This gives us a
short exact sequence
$$0 \to TM \stackrel{\inc}{\longrightarrow} T^2M \stackrel{\tensorq \, \,}{\longrightarrow} \Sym^2 TM \to 0,$$
where $\tensorq(L) =: \tensorq_L$ is the symmetric bilinear tensor on
$T^*M$ defined by the property
\begin{equation} \label{qform} \tensorq_L(df, dg) = \frac{1}{2}
  \left(L(fg) - f Lg - g Lf\right), \qquad \text{for any } f,g \in
  C^\infty(M).\end{equation} In local coordinates, we can write $\tensorq_L(df,dg) =
\sum_{i,j=1}^{n+\nu} L_{ij} \frac{\partial f}{\partial x_i}
\frac{\partial g}{\partial x_j}$ relative to the representation of $L$ in \eqref{localExpression}.

Let $\tensorh^*$ be the co-metric corresponding to a sub-Riemannian
structure $(\calH, \tensorh)$. Then any operator $L$ satisfying
$\tensorq_L = \tensorh^*$ can locally be written as
$$L = \sum_{i=1}^n A_i^2 + A_0$$
where $A_0$ is a vector field and $A_1, \dots, A_n$ is a local
orthonormal basis of $\calH$. From H\"ormander's celebrated result
\cite{Hor67}, we know that any such operator is hypoelliptic when
$\calH$ is bracket-generating. We consider two examples where a choice
of extra structure on $(M, \calH, \tensorh)$ gives a differential
operator of this type.

Let $\vol$ be a volume form on $M$. We define the {\it sub-Laplacian
  relative to} $\vol$ as the operator $\srL$ given by
$$\srL f: = \dv \shh df,$$
where $\dv A$ is defined by $\calL_A \vol = (\dv A) \vol$. Any such
operator satisfies
$$\int_M g \srL f \, \dvol = \int_M f \srL g \, \dvol$$
for any pair of functions $f,g \in C_c^\infty(M)$ of compact
support. Since $L$ is also hypoelliptic, it has a smooth, symmetric
heat kernel with respect to $\vol.$ This is the most common way of
defining the sub-Laplacian.

We like to introduce an alternative notion of sub-Laplacian. Rather
than choosing a volume, we will choose a complement $\calV$ to
$\calH$, i.e. a subbundle $\calV$ such that $TM = \calH \oplus
\calV$. This choice of complement gives us projections $\pr_{\calH}$
and $\pr_{\calV}$ to respectively~$\calH$ and~$\calV$. A Riemannian
metric $\tensorg$ on $M$ is said to {\it tame} $\tensorh$ if
$\tensorg|_{\calH} = \tensorh.$ Consider any Riemannian metric
$\tensorg$ that tames $\tensorh$ and makes $\calV$ the orthogonal
complement of $\calH$. Let $\nabla$ be the Levi-Civita connection of
$\tensorg$. It is simple to verify that for any pair of horizontal
vector fields $A,Z \in \Gamma(\calH)$, $\pr_{\calH} \nabla_{A} Z$ is
independent of~$\tensorg|_{\calV}$. This fact allows us to define a
second order operator $\roughL$ which we call \emph{the sub-Laplacian
  with respect to} $\calV$. There are several ways to introduce this
operator. We have chosen to define it by using a connection $\rnabla$
which will be helpful for us later. Corresponding to the Riemannian
metric $\tensorg$ and the orthogonal splitting $TM = \calH
\oplus_{\perp} \calV$, we introduce the connection
\begin{align} \label{rnabla}
  \rnabla_A Z := & \pr_{\calH} \nabla_{\pr_{\calH} A} \pr_{\calH} Z + \pr_{\calV} \nabla_{\pr_{\calV} A} \pr_{\calV} Z  \\
  \nonumber & + \pr_{\calH} [ \pr_{\calV} A, \pr_{\calH} Z] +
  \pr_{\calV} [ \pr_{\calH} A, \pr_{\calV} Z].
\end{align}
\begin{definition} \label{def:srL} Let $\calV$ be a complement of
  $\calH$ corresponding to the projection $\pr_{\calH}$, i.e. $\calV =
  \ker \pr_{\calH}$. Then the sub-Laplacian with respect to $\calV$ is
  the operator
$$\roughL f := \tr_{\calH} \rnabla^2_{\times, \times} f, \quad f \in C^\infty(M),$$
where $\rnabla^2_{A,Z} = \rnabla_{A} \rnabla_Z - \rnabla_{\rnabla_A
  Z}$ is the Hessian of $\rnabla$.
\end{definition}
We remark that the definition only depends on the value of $\rnabla_A
Z$ when both $A$ and $Z$ take values in $\calH$. This is
illustrated by the fact that locally
\begin{equation} \label{localRoughL} \roughL f = \sum_{i=1}^n A_i^2 f
  + \sum_{i, j=1}^n \tensorh( \pr_{\calH} \nabla_{A_i} A_j, A_i) A_j
  f\end{equation} where $A_1, \dots, A_n$ is a local orthonormal
basis of $\calH$.  The operator $\roughL$ is hypoelliptic and will
have a smooth heat kernel with respect to any volume form.
Two different choices of complement may have the same
sub-Laplacian, see Section~\ref{sec:2Vcomp}.

In what follows, whenever we have a chosen complement $\calV$, we will
refer to it as {\it the vertical bundle} and its vectors and vector
fields as {\it vertical}.

\begin{remark} \label{re:regularH} The horizontal bundle $\calH$ is
  called \emph{equiregular} if there exist a flag of subbundles
$$\calH = \calH^1 \subseteq \calH^2 \subseteq \calH^3 \subseteq \cdots$$
such that $$\calH^{k+1} = \spn\left\{ Z|_x,\, [A,Z]|_x \, \colon \, Z \in
\Gamma(\calH^k), \ A \in \Gamma(\calH),\ x \in M \right\}.$$
Even if $\calH$ is bracket-generating, it is not necessarily equiregular. 
We emphasize that each $\calH^k$ is required to be a subbundle, 
and so must have constant rank. The smallest
integer $r$ such that $\calH^r = TM$ is called \emph{the step} of
$\calH$. If $(\calH, \tensorh)$ is a sub-Riemannian structure on $M$
with $\calH$ equiregular, then there exist a canonical choice of volume
form on $M$ called Popp's measure. For construction, see
Section~\ref{sec:PrivMetric} or see~\cite{ABGR09} for a more detailed
presentation.
\end{remark}

\subsection{Lifting property of the sub-Laplacian defined relative to a complement} \label{sec:LiftedLaplacian}
Let $\pi: M \to B$ be a surjective submersion between connected
manifolds $M$ and $B$. \emph{The vertical bundle of} $\pi$ is the
subbundle $\calV := \ker \pi_*$ of $TM$. \emph{An Ehresmann connection
  on} $\pi$ is a splitting $h$ of the short exact sequence
$$0 \longrightarrow  \calV = \ker \pi_* \longrightarrow\xymatrix{TM \ar[r]_{\pi_*} & \pi^* TB \ar@/_/[l]_h} \longrightarrow 0.$$ 
This map $h$ is uniquely determined by $\calH = \text{image } h$,
which is a subbundle of $TM$ satisfying $TM = \calH \oplus
\calV$. Hence, we refer to such a subbundle $\calH$ as an Ehresmann connection as
well. The image of an element $(x, \check{v}) \in \pi^*TB$ under $h$
is called the horizontal lift of $\check{v}$ to $x$, and denoted $h_x
\check{v}$. Similarly, for any vector field $\check{A}$ on $B$, we
have a vector field $h\check{A}$ on $M$ defined by $x \mapsto h_x
\check{A}|_{\pi(x)}$.

We can extend the notion of horizontal lifts to second order vectors
and differential operators. If $B$ and $M$ are two manifolds, then a
linear map $\varphi: T^2_b B \to T^2_x M$ is called a {\it Schwartz
  morphism} if $\varphi(T_bM) \subseteq T_xM$ and the following diagram
commutes
$$\xymatrix{
  0 \ar[r] & T_xM \ar[r]^{\inc} & T^2_xM \ar[r]^{\tensorq} & \Sym^2 T_xM \ar[r] & 0 \\
  0 \ar[r] & T_bB \ar[r]^{\inc} \ar[u]_{\varphi|_{T_bB}} & T^2_bB
  \ar[r]^{\tensorq} \ar[u]_{\varphi} & \Sym^2 T_bB \ar[r] \ar[u]_{\varphi|_{T_bB}
    \otimes \varphi|_{T_bB} } & 0 } $$
with $\tensorq$ defined as in \eqref{qform}. We remark that any linear
map $\varphi\colon T_b^2B \to T^2_x M$ is a Schwartz morphism if and only if $\varphi =
f_* |_{T^2_bB}$ for some map $f\colon B \to M$ with $f(b) = x$ (see
e.g. \cite[p. 80]{Eme89}). Let $\pi\colon M \to B$ be a surjective
submersion. A $2$-\emph{connection on} $\pi$ is then a splitting $h^S$
of the short exact sequence
$$0 \longrightarrow \ker \pi_* \longrightarrow \xymatrix{T^2M
  \ar[r]_{\pi_*} & \pi^* T^2B \ar@/_/[l]_{h^S}} \longrightarrow 0$$
 such that $h^S$ is a Schwartz morphism at any point.

For any choice of Ehresmann connection $h$ on $\pi$, we can construct
a corresponding 2-connection $h^S$ uniquely determined by the
following two requirements (see e.g. \cite[pp 82--83]{Mey81}).
\begin{enumerate} [$\bullet$]
\item $h^S|_{TM} = h$,
\item $h^S\{ \check{A}, \check{Z}\} = \{ h\check{A}, h\check{Z}\}$
  where $\{\check{A}, \check{Z}\} = \frac{1}{2} (\check{A}\check{Z} +
  \check{Z} \check{A})$ is the skew-commutator and $\check{A},
  \check{Z} \in \Gamma(TB)$. Equivalently, $h^S(\check{A} \check{Z}) =
  h\check{A} h\check{Z} - \frac{1}{2} \pr_{\calV} [h\check{A},
  h\check{Z}].$
\end{enumerate}
Using this 2-connection, we can define horizontal lifts of second
order operators on~$B$. We then have the following way to interpret
the sub-Laplacian with respect to a complement.

\begin{proposition} \label{prop:liftingL} Let $\widecheck{\tensorg}$
  be a Riemannian metric on $B$ with Laplacian
  $\widecheck{\Delta}$. Relative to an Ehresmann connection $\calH$ on
  $\pi$, define a sub-Riemannian structure $(\calH, \tensorh)$ by
  $\tensorh = \pi^* \tensorg|_{\calH}$. Then $h^S \widecheck{\Delta}=
  \roughL$ where $\roughL$ is the sub-Laplacian of~$\calV = \ker
  \pi_*$.
\end{proposition}

In particular, for any $f \in C^\infty(B)$, we have $\roughL (f \circ \pi) = (\widecheck{\Delta} f) \circ \pi$.
A submersion $\pi:(M, \tensorg) \to (B, \widecheck{\tensorg})$ between two Riemannian manifolds such that
$$\tensorg|_{\calH} = \pi^* \widecheck{\tensorg}|_{\calH}, \qquad  \calH
= (\ker \pi_*)^\perp$$ is called \emph{a Riemannian
  submersion}. The sub-Riemannian manifolds of
Proposition \ref{prop:liftingL} can hence be considered as the result
of restricting the metric on the top space in a Riemannian submersion
to its horizontal subbundle.

\begin{proof}[of Proposition {\upshape\ref{prop:liftingL}}]
  Let $\tensorg$ be a Riemannian metric on $M$ satisfying
  $\tensorg|_{\calH} = \tensorh$ and $\calH^\perp = \calV$. Let $\check{A}_1,\dots, \check{A}_n$ be any
  local orthonormal basis of $TB$. Then the Laplacian can be written
  as
$$\widecheck{\Delta} = \sum_{i=1}^n \check{A}_i^2 + \sum_{i,j=1}^n \widecheck{\tensorg}(\widecheck{\nabla}_{\check{A}_i} \check{A}_j, \check{A}_i) \check{A}_j$$
where $\widecheck{\nabla}$ is the Levi-Civita connection of
$\widecheck{\tensorg}$. However, since $\tensorg(h\check{A}_i,
h\check{A}_j) = \widecheck{\tensorg}(\check{A}_i, \check{A}_j)$ and
since $\pr_{\calH} [h\check{A}_i, h\check{A}_j] = h[\check{A}_i,
\check{A}_j]$, we obtain
$$\widecheck{\tensorg}(\widecheck{\nabla}_{\check{A}_i} \check{A}_j, \check{A}_i) = \tensorg(\rnabla_{h\check{A}_i} h\check{A}_j, h\check{A}_i), \qquad i,j =1,2, \dots, n.$$
The result follows from \eqref{localRoughL} and the fact that
$h\check{A}_1, \dots, h\check{A}_n$ forms a local orthonormal basis
of~$\calH$.\qed
\end{proof}

Since the proof of Proposition~\ref{prop:liftingL} is purely local, it
also holds on Riemannian foliations. To be more specific, a subbundle
$\calV$ of $TM$ is \emph{integrable} if $[V_1, V_2]$ takes its values in
$\calV$ whenever $V_1,V_2 \in \Gamma(\calV)$ are vertical vector
fields. From the Frobenius theorem, we know that there exists a
foliation $\calF$ of $M$ consisting of immersed submanifolds of
dimension $\nu = \rank \calV$ such that each leaf is tangent to $\calV$.

A Riemannian metric $\tensorg$ on $M$ with a foliation induced by
$\calV$, is called \emph{bundle-like} if $\calV^{\perp} = \calH$ and
for any $A \in \Gamma(\calH)$ and $V \in \Gamma(\calV)$, we have
$(\calL_V \tensorg)(A,A) = 0$. Intuitively, one can think of a bundle-like metric
$\tensorg$ as a metric where $\tensorg|_{\calH}$ does ``not change''
in vertical directions. A foliation $\calF$ of a Riemannian manifold is called
\emph{Riemannian} if the metric is bundle-like with respect to $\calF$. Such a manifold
locally has the structure of a Riemannian submersion, that is, any
point has a neighborhood $U$ such that $\pi: U \to B:= U/(\calF|U)$
can be considered as a smooth submersion of manifolds. The subbundle $\calH|_U$ is an
Ehresmann on $\pi$ and $B$ can be given a Riemannian metric
$\widecheck{\tensorg}$ such that $\pr_{\calH}^* \tensorg = \pi^*
\widecheck{\tensorg}$, see~\cite{Rei59}. If we define a sub-Riemannian
structure $(\calH, \tensorh)$ on~$M$ with $\tensorh =
\tensorg|_{\calH}$, then restricted to each sufficiently small
neighborhood $U$, the sub-Laplacian $\roughL$ of $\calV$ is equal to
$h^S \widecheck{\Delta}$ were $\widecheck{\Delta}$ is the Laplacian on
$U/(\calF|U)$.

\begin{remark}
  By modifying the proof of Proposition~\ref{prop:liftingL} slightly,
  we can also get the following stronger statement: Let $(B, \calH_1,
  \tensorh_1)$ be a sub-Riemannian manifold and let $\pi: M \to B$ be
  a submersion with an Ehresmann connection $\calE$. Define a
  subbundle $\calH_2$ on $M$ as the horizontal lifts of all vectors in
  $\calH_1$ with respect to $\calE$ and let $\tensorh_2 = \pi^*
  \tensorh_1|_{\calH_2}$ be the lifted metric. Let $\calV_1$ be a choice of complement of
  $\calH_1$ and define $\calV_2$ as the direct sum of the horizontal
  lift of $\calV_1$ and $\ker \pi_*$. Then, if
  $\Delta_{\calH_j}^\prime$ is the sub-Laplacian with respect to
  $\calV_j, j=1,2$, we have $\Delta_{\calH_2}^\prime =
  h^S\Delta_{\calH_1}^\prime$, where $h^S$ is also defined with respect to $\calE$.
\end{remark}

\subsection{Comparison between the sub-Laplacian of a complement and a
  volume form} \label{sec:CompVSvol} Let $(M, \calH, \tensorh)$ be a sub-Riemannian manifold
and let $\tensorg$ be a Riemannian metric taming~$\tensorh$. Let
$\srL$ be the sub-Laplacian defined with respect to the volume form of
$\tensorg$ and let $\roughL$ be defined relative to the complement $\calH^\perp =
\calV$. We introduce a vector field $N$ by formula
$$\srL = \roughL - N.$$
It can then be verified that $N$ is horizontal and can be
defined by the relation
\begin{equation} \label{MeanCurv} \tensorg(A, N) = - \frac{1}{2} \tr_{\calV}
  (\calL_{\pr_{\calH} A} \tensorg)(\times ,
  \times).\end{equation} 
In order for $\roughL$ to be the sub-Laplacian with respect to some volume form, we must have $N = - \shh d\phi$ for some function $\phi \in C^\infty(M)$. Indeed, if $\dv A$ denotes the divergence of a vector field $A$ with respect to $\vol$, then $\dv A + d\phi(A)$ is its divergence with respect to $e^\phi \vol$. It follows that the sub-Laplacian of $e^\phi \vol$ is given as $\srL f + (\shh d\phi) f.$

\begin{remark}
  Let $(M, \tensorg)$ be a Riemannian manifold with a foliation given
  by an integrable subbundle $\calV$. Assume that $\tensorg$ is
  bundle-like relative to $\calV$. Let $\calH$ be the orthogonal
  complement of $\calV$ and define $\tensorh =
  \tensorg|_{\calH}$. Write $\vol$ for the volume form of
  $\tensorg$. Let $\srL$ and $\roughL$ be the sub-Laplacian of $(M,
  \calH, \tensorh)$ relative to respectively $\vol$ and $\calV$. Then
  $N = \roughL - \srL$ is the mean curvature vector field of the
  leaves of the foliations by \eqref{MeanCurv}. Hence, the operators
  $\srL$ and $\roughL$ coincide in this case if and only if the leafs of
  the foliation are minimal submanifolds.
\end{remark}

\subsection{Diffusion of $\roughL$}
Let $L$ be any section of $T^2M$ with $\tensorq_L$ being positive
semi-definite and let $x \in M$ be any point. Then, by \cite[Theorems 1.3.4 and
1.3.6]{Hsu02} there exists an $L$-diffusion $X_t = X_t(x)$
satisfying $X_0 = x$, unique in law, defined up to some
explosion time $\tau = \tau(x)$. {\it An $L$-diffusion} $X_t$ is an $M$-valued semimartingale up to some stopping time
$\tau$ defined on some filtered probability space $(\Omega,
\scrF_{\! \newbulletsub}, \prob)$, such that for any $f \in C^\infty(M)$,
$$M_t^f := f( X_t) - f( X_0) - \int_0^t L f( X_s) ds, \quad 0 \leq t < \tau,$$
is a local martingale up to $\tau$. We will always assume that
$\tau$ is maximal, so that $\tau$ is the explosion time,
i.e. $\{ \tau < \infty\} \subseteq \{ \lim_{t \uparrow \tau} X_t = \infty\} \text{ almost surely.}$

Let $\roughL$ be the sub-Laplacian defined with respect to a choice of
complement~$\calV$. Let~$\tensorg$ be a Riemannian metric such that
$\calH^\perp = \calV$ and $\tensorg|_{\calH} = \tensorh$. Define
$\rnabla$ on as in~\eqref{rnabla}. To simplify our presentation, we
will assume that $\rnabla \tensorg = 0$. See Remark~\ref{re:znabla}
for the general case. For a given point $x \in M$, let $\gamma(t)$ be
any smooth curve in $M$ with $\gamma(0) = x$. Let $\phi_1, \dots,
\phi_n$ and $\psi_1, \dots, \psi_\nu$ be orthonormal bases for
respectively $\calH|_x$ and $\calV|_x$. Parallel transport of
such bases remain
orthonormal bases from our assumption $\rnabla \tensorg =0$.

Define $\Oct(n) \to \Oct(\calH) \to M$ as the bundle of orthonormal
frames of $\calH$, and define $\Oct(\nu) \to \Oct(\calV) \to M$
similarly. Let
$$\Oct(n) \times \Oct(\nu) \to \Oct(\calH) \odot \Oct(\calV) \stackrel{\pi}{\to} M$$
denote the product bundle. We can then define an Ehresmann connection
$\calE^{\rnabla}$ on $\pi$ such that a curve $(\phi(t), \psi(t)) =
(\phi_1(t), \dots, \phi_n(t), \psi_1(t), \dots, \psi_\nu(t))$ in
$\Oct(\calH) \odot \Oct(\calV)$ is tangent to $\calE^{\rnabla}$ if and
only if each $\phi_j(t)$, $1 \leq j \leq n$, and $\psi_s(t)$, $1 \leq s
\leq \nu$, is parallel along $\gamma(t) = \pi(\phi(t), \psi(t)).$

Define vector fields $\widetilde A_1, \dots, \widetilde A_n$ on
$\Oct(\calH) \odot \Oct(\calV)$ by $\widetilde A_j |_{\phi, \psi} =
h_{\phi, \psi} \phi_j$ where the horizontal lift is with respect to
$\calE^{\rnabla}$. For any $x \in M$ and $(\phi, \psi) \in \Oct(\calH)
\odot \Oct(\calV)|_x$, consider the solution $\Phi_t$ of the
Stratonovich SDE up to explosion time $\tau$,
$$d\Phi_t = \sum_{j=1}^n \widetilde A_j|_{\Phi_t} \circ dW_t^j, \qquad \Phi_0 = (\phi, \psi),$$
where $W = (W^1, \dots, W^n)$ is a Brownian motion
in $\real^n$ with $W_0 = 0.$ It is then simple to verify that
$\Phi$ is a $\frac{1}{2} h^S \roughL$-diffusion  since $h^S
\roughL = \sum_{i=1}^n \widetilde A_j^2$ if we consider the 2-connection $h^S$
induced by the Ehresmann connection $\calE^{\rnabla}$. This shows that $X_t =
\pi(\Phi_t)$ is an $\frac{1}{2} \roughL$-diffusion on $M$ with
$X_0 = x$. Note that $\tau$ will be the explosion time of
$X_t$ as well by \cite{Shi82}.

\begin{remark} \label{re:znabla} If $\rnabla \tensorg \neq 0$, we can
  instead use the connection
  \begin{equation} \label{znabla} \znabla_A Z := \pr_{\calH} \nabla_A
    \pr_{\calH} Z + \pr_{\calV} \nabla_A \pr_{\calV} Z, \qquad A,Z \in
    \Gamma(TM).\end{equation} It clearly satisfies $\znabla \tensorg=
  0$, preserves the horizontal and vertical bundle under parallel
  transport and has $\rnabla_A Z = \znabla_A Z$ for any pair
  horizontal vector fields $A,Z \in \Gamma(\calH)$. The definition of
  $\srL$ in Definition~\ref{def:srL} hence remains the same if we
  replace~$\rnabla$ with~$\znabla$. The reason we will prefer to use
  $\rnabla$ is the property given in
  Lemma~\ref{lemma:basicsrnabla}~(a) which fails for~$\znabla$.
\end{remark}

\begin{remark}
  Instead of using the lift to $\Oct(\calH) \odot \Oct(\calV)$, we
  could have considered the full frame bundle $\Oct(TM)$ and
  development with respect to $\rnabla$ or $\znabla$, see
  e.g. \cite[Section 2.3]{Hsu02}. The diffusion of $\frac{1}{2}
  \roughL$ then has the Brownian motion in an $n$-dimensional subspace
  of $\real^{n+\nu}$ as its anti-development, where the subspace
  depends on the choice of initial frame.
\end{remark}

\begin{remark} If $\roughL$ is symmetric with respect to some volume form $\vol$, the
the following observation made in \cite[Theorem 4.4]{BaWa14} guarantees us that the diffusion $X_t$
has infinite lifetime, i.e. $\tau = \infty$ a.s. Let $\tensorg$ be any
Riemannian metric on $M$ taming~$\tensorh$ with corresponding volume form $\vol$. Assume that $\tensorg$ is
complete with (Riemannian) Ricci curvature bounded from below. Note
that if $\metricd_{\tensorg}$ is the metric of $\tensorg$ and
$\metricd_{cc}$ is the Carnot-Carath\'eodory of $\tensorh$, then
$\metricd_{cc}(x, y) \geq \metricd_{\tensorg}(x, y)$ for any $(x,
y) \in M \times M$. Hence, $B_r(x) \subseteq B_r^{\tensorg}(x)$
where $B_r(x)$ and $B_r^{\tensorg}(x)$ are the balls of respectively
$\metricd_{cc}$ and $\metricd_{\tensorg}$, centered at $x$ with radius
$r$. By the Riemannian volume comparison theorem, we have
$$\vol(B_r(x)) \leq \vol(B_r^{\tensorg}(x)) \leq C_1 e^{C_2r}$$
for some constants $C_1, C_2$. In conclusion, $\int_0^\infty \frac{r}{\log
  \vol(B_r(x))} dr = \infty$ and so \cite[Theorem 3]{Stu94} tells us that
 $\frac{1}{2} \roughL$-diffusions $X_t$ starting at a point $x \in M$ has
infinite lifetime.
\end{remark}

\section{Riemannian foliations and the curvature-dimension
  inequality} \label{sec:CurvatureDimension}
\subsection{Riemannian foliations and the geometry of
  $\rnabla$} \label{sec:GeometryRnabla} Let $(M, \calH, \tensorh)$ be
a sub-Riemannian manifold with a complement $\calV$. Let $\roughL$ be
the sub-Laplacian relative to $\calV$. In order to introduce a
curvature-dimension inequality for $\roughL$, we will need to choose a
Riemannian metric on $M$ which tame $\tensorh$ and makes $\calH$ and
$\calV$ orthogonal. Choose a metric tensor $\tensorv$ on $\calV$ to
obtain a Riemannian metric $\tensorg = \pr_{\calH}^* \tensorh +
\pr_{\calV}^* \tensorv$.

We make the following assumptions on $\calV$. We want to consider the
specific case when $\calV$ is integrable and satisfies
\begin{equation} \label{mpintegral} \calL_V (\pr_{\calH}^* \tensorh) =
  0 \quad \text{ for any } V \in \Gamma(\calV).\end{equation} Then
$\tensorg = \pr_{\calH}^* \tensorh + \pr_{\calV}^* \tensorv$ is
bundle-like for any choice of $\tensorv$, giving us a Riemannian
foliation as defined in Section~\ref{sec:LiftedLaplacian}. Since this property is
independent of $\tensorg|_{\calV}$, we introduce the
following definition.
\begin{definition}
  An integrable subbundle $\calV$ is called a metric-preserving
  complement to the sub-Riemannian manifold $(M, \calH, \tensorh)$ if
  $TM = \calH \oplus \calV$ and \eqref{mpintegral} hold.
\end{definition}
In the special case when the foliation $\calF$ of $\calV$ gives us a
submersion $\pi: M \to B= M/\calF$ with $\calH$ as an Ehresmann
connection on $\pi$, \emph{the curvature} of $\calH$ is a
vector-valued two-form $\calR \in \Gamma(\bigwedge^2 T^*M \otimes TM)$
defined by
\begin{equation} \label{EhresmannCurv} \calR(A,Z) := \pr_{\calV} [
  \pr_{\calH} A, \pr_{\calH} Z], \quad A,Z \in
  \Gamma(TM).\end{equation} This curvature measures how far $\calH$ is
from being a flat connection, i.e. an integrable subbundle. We will
call $\calR$ given by formula \eqref{EhresmannCurv} the curvature of
$\calH$ even when $\calV$ does not give us a submersion globally.

Define $\rnabla$ relative to $\tensorg$ as in \eqref{rnabla}. The
following properties are simple to verify.

\begin{lemma} \label{lemma:basicsrnabla} Let $\tensorg$ be a
  Riemannian metric and let $\calV$ be an integrable subbundle of $TM$
  with orthogonal complement $\calH$. Define $\rnabla$ relative to
  $\tensorg$ and the splitting $TM = \calH \oplus \calV$. Write
  $\tensorh$ and $\tensorv$ for the restriction of $\tensorg$ to
  respectively $\calH$ and $\calV$.
  \begin{enumerate}[\rm(a)]
  \item Let $Z\in \Gamma(TM)$ be an arbitrary vector field. If
    $A\in \Gamma(\calH)$ is horizontal, both $\rnabla_{A} Z$ and
    $\rnabla_Z A$ only depends on $\tensorh$ and the splitting $TM =
    \calH \oplus \calV$. They are independent of~$\tensorv$. Similarly,
    if $V$ is vertical, then $\rnabla_V Z$ and $\rnabla_Z V$ are
    independent of~$\tensorh$.
  \item The torsion of $\rnabla$ is given as $T^{\rnabla}(A,Z) =
    -\calR(A,Z).$
  \item $\calV$ is a metric-preserving complement of $(M, \calH,
    \tensorh)$ if and only if $$(\rnabla_A \tensorg)(Z,Z) =
    (\calL_{\pr_{\calH} A} \tensorg)(\pr_{\calV} Z, \pr_{\calV} Z).$$
    Equivalently, $\calV$ is metric-preserving if and only if $\rnabla
    \tensorh^* = 0$.
  \item If $\calV$ is metric-preserving, then
$$(\rnabla_{Z_1} \tensorg)(Z_2,Z_3) = - 2\tensorg(Z_1, \II(\pr_{\calV} Z_2, \pr_{\calV} Z_3)),$$
where $\II$ is the second fundamental form of the foliation of
$\calV$.
\end{enumerate}
\end{lemma}
Recall that when $\calV$ is metric-preserving, $\tensorg$ is
bundle-like. We write down the basic properties of the
curvature~$R^{\rnabla}$ of~$\rnabla$ when $\calV$ is
metric-preserving.
\begin{lemma} \label{lemma:curvaturernabla} Let $Z_1,Z_2 \in \Gamma(TM)$
  be arbitrary vector fields and let $A \in \Gamma(\calH)$ be a
  horizontal vector field. Then
  \begin{enumerate}[\rm(a)]
  \item $\tensorg(R^{\rnabla}(Z_1,Z_2) A,A) = 0$,
  \item $ \tensorg(R^{\rnabla}(A, Z_1) Z_2, A) - \tensorg(R^{\rnabla}(A,
    Z_2) Z_1, A) = 0$. 

In particular, $\tensorg(R^{\rnabla}(A, \pr_{\calV} Z_2) Z_1, A) =0.$
\end{enumerate}
\end{lemma}

\begin{proof} The statement in (a) holds since $(\rnabla
  \tensorg)(\pr_{\calH} \newbullet ,\pr_{\calH} \newbullet) = 0$. For
  the identity in~(b), we recall the first Bianchi identity for
  connections with torsion,
$$\circlearrowright R^{\rnabla}(A, Z_1) Z_2 
= - \circlearrowright T^{\rnabla}(A, T^{\rnabla}(Z_1,Z_2)) +
\circlearrowright (\rnabla_{A} T^{\rnabla})(Z_1, Z_2),$$ where
$\circlearrowright$ denotes the cyclic sum.
This means that
\begin{align*}
  \tensorg&(\circlearrowright R^{\rnabla}(A, Z_1) Z_2 , A) = \tensorg( R^{\rnabla}(A, Z_1) Z_2 , A) - \tensorg( R^{\rnabla}(A, Z_2) Z_1 , A) \\
  &= \tensorg\left(- \circlearrowright T^{\rnabla}(A,
    T^{\rnabla}(Z_1,Z_2)) + \circlearrowright (\rnabla_{A} T^{\rnabla})(Z_1,
    Z_2), A\right) = 0.
\end{align*}\qed
\end{proof}

We will use the fact that we have a clear idea of what Ricci
curvature is on a Riemannian manifold, to introduce a corresponding
tensor on a sub-Riemannian manifold with a metric-preserving
complement $\calV$.
\begin{proposition} \label{prop:RicH} Introduce a tensor
  $\RicH \in \Gamma(T^*M^{\otimes 2})$ by
$$\RicH(Z_1,Z_2) = \tr R^{\rnabla}(\pr_{\calH} \newbullet, Z_2) Z_1.$$
Then
\begin{enumerate}[\rm(a)]
\item $\RicH$ is symmetric and $\calV \subseteq \ker \RicH$.
\item $\RicH$ is independent of choice of metric $\tensorv$ on
  $\calV$.
\item Let $\calF$ be the foliation induced by $\calV$. Let $U$ be any
  neighborhood of $M$ such that the quotient map $\pi\colon U \to
  B:=U/(\calF|U)$ is a smooth submersions of manifolds. Let
  $\widecheck{\Ric}$ be the Ricci curvature on $B$ with respect to the
  induced Riemannian structure. Then ${\RicH}|_U = \pi^*
  \widecheck{\Ric}$.
\end{enumerate}
\end{proposition}
\begin{proof}
 Since $\rnabla$ preserves both the vertical and horizontal bundle, and by means of
  Lemma~\ref{lemma:curvaturernabla}~(b), we know that $\RicH = \pr_{\calH}^* \RicH$. It is symmetric by Lemma~\ref{lemma:curvaturernabla}~(b), which completes the proof of the statement in (a). The statement (b) can be verified using the definition of the Levi-Civita connection.

  To prove (c), let $\check{A}_1, \dots, \check{A}_n$ be any local
  orthonormal basis on $B$. Note that $$[h\check{A}_i, h\check{A}_j] =
  h[\check{A}_i, \check{A}_j] + \calR(h\check{A}_i,
  h\check{A}_j).$$ Also, for any $V \in \Gamma(\calV)$, $\rnabla_{V}
  h\check{A}_i = \pr_{\calH} [V,h\check{A}_i] =0$ since $h\check{A}_i$
  and $V$ are $\pi$-related to respectively $\check{A}_i$ and the
  zero-section of $TB$. Finally, recall that $\rnabla_{hA_i} hA_j = h \widecheck{\nabla}_{A_i} A_j$ from the proof of Proposition~\ref{prop:liftingL}. For any $j,k$,
  \begin{align*}
    \sum_{i=1}^n &\tensorg(R^{\rnabla}(h\check{A_i}, h\check{A}_j)h\check{A}_k, h\check{A}_i) \\
    &=  \sum_{i=1}^n \tensorg\left(\left[ \rnabla_{h\check{A}_i} , \rnabla_{h\check{A}_j} \right] h \check{A}_k - \rnabla_{h[\check{A}_i, \check{A}_j]} h \check{A}_k - \rnabla_{\calR(h\check{A}_i, h\check{A}_j)} h\check{A}_k, h\check{A}_i \right) \\
    &= \sum_{i=1}^n \widecheck{\tensorg}\left(\left[
        \widecheck{\nabla}_{\check{A}_i} ,
        \widecheck{\nabla}_{\check{A}_j} \right] \check{A}_k -
      \widecheck{\nabla}_{[\check{A}_i, \check{A}_j]} \check{A}_k ,
      \check{A}_i \right) = \widecheck{\Ric}(\check{A}_k,
    \check{A}_j).
  \end{align*}
  It follows that $\RicH(h\check{A}_k, h\check{A}_j) =
  \widecheck{\Ric}(\pi_* h\check{A}_k, \pi_* h\check{A}_j)$, and hence
  the same holds for any pair of vector fields $Z_1,Z_2$.\qed
\end{proof}

\subsection{A generalized curvature-dimension
  inequality} \label{sec:GenCD} For any symmetric bilinear tensor
$\tensors^* \in \Gamma(\Sym^2 TM)$, we associate a symmetric map
$\sfGamma^{\tensors^*}$ of smooth functions by
$$\begin{array}{rccc} \sfGamma^{\tensors^*}\colon&   C^\infty(M) \times C^\infty(M) & \to & C^\infty(M) \\
  & (f,g) & \mapsto & \tensors^*(df, dg).
\end{array} $$
By Leibniz identity, we have relation $\sfGamma^{\tensors^*}(f, g \phi) = g\sfGamma^{\tensors^*}(f,\phi) + \phi\sfGamma^{\tensors^*}(f,g)$ for arbitrary smooth functions $f,g,\phi$. 
Relative to some second order operator $L \in \Gamma(T^2M)$, we define
$$\sfGamma_{2}^{\tensors^*}(f,g) := \frac{1}{2} \left(L \sfGamma^{\tensors^*}(f,g) - \sfGamma^{\tensors^*}(L f, g) - \sfGamma^{\tensors^*}(f, L g) \right).$$
To simplify notation, we will write $\sfGamma^{\tensors^*}(f,f) =
\sfGamma^{\tensors^*}(f)$ and $\sfGamma^{\tensors^*}_2(f,f) =
\sfGamma^{\tensors^*}_2(f)$.

Let $\tensorh^* = \tensorq_L$ where $\tensorq$ is defined as in
\eqref{qform}. Assume that $\tensorh^*$ is positive semi-definite and
let $\tensorv^*\in \Gamma(\Sym^2 TM)$ be another chosen positive
semi-definite section. Then $L$ is said to satisfy {\it a generalized
  curvature-dimension inequality} with parameters $n,\rho_{1},
\rho_{2,0}$ and $\rho_{2,1}$ if
\begin{align} \tag{{\sf CD*}} \label{CDstar}
  \sfGamma^{\tensorh^*}_2(f) + \ell \sfGamma_2^{\tensorv^*}(f) \geq &
  \frac{1}{n} (Lf)^2 + \left( \rho_{1} - \frac{1}{\ell}\right)
  \sfGamma^{\tensorh^*}(f) + (\rho_{2,0} + \rho_{2,1} \ell )
  \sfGamma^{\tensorv^*}(f),\end{align} for any $\ell >0$. We
include the possibility of $n = \infty$. Any such inequality implies
$\sfGamma_2^{\tensorv^*}(f) \geq \rho_{2,1} \sfGamma^{\tensorv^*}(f)$
by dividing both sides with $\ell$ and letting it go to infinity.

Let $(M, \calH, \tensorh)$ be a sub-Riemannian manifolds with an
integrable complement $\calV$ that is also metric-preserving. Choose a
metric $\tensorv$ on $\calV$, and define $\rnabla$ with respect to the
corresponding Riemannian metric $\tensorg$. Let $\tensorv^*$ be the
co-metric corresponding to $\tensorv$. Using the properties of
$\rnabla$, we are ready to present our generalized curvature-dimension
inequality. We will make the following assumptions on $(M, \calH,
\tensorh)$.
\begin{enumerate}[(i)]
\item Let $\calR$ be the curvature of $\calH$ relative to the
  complement $\calV$. Assume that the length of $\calR$ is bounded on
  $M$ and define $\McalR < \infty$ as the minimal number such that
$$ \|\calR(v,\newbullet)\|_{\tensorg^* \otimes \tensorg} \leq \McalR \|\pr_{\calH} v\|_{\tensorg}, \quad \text{for any }
v \in TM.$$
By replacing $\tensorv$ with $\McalR^{-2} \tensorv$, we may assume that $\McalR =1$. 
From now on, we will work with the vertical metric $\tensorv$ normalized in this way.
\item Let $\RicH$ be defined as in Proposition~\ref{prop:RicH}. Assume
  that $\RicH$ has a lower bound $\rRicH$, i.e. for every $v \in TM$,
  we have $$\RicH(v,v) \geq \rRicH \|\pr_{\calH} v\|_{\tensorh}^2.$$
\item Assume that the length of the tensor $\rnabla
  \tensorg^*$ ($= \rnabla \tensorv^* $) is bounded.  Write
$$\MII = \sup_{M} \big\| \rnabla_{\!\newbulletsub}^{\mathstrut} \tensorv^* (\newbullet,\newbullet) \big\|_{\tensorg^* \otimes \Sym^2 \tensorg^*}.$$
Define also
$$(\roughL \tensorv^*)(p,p) = \tr_{\calH} (\rnabla_{\times ,\times}^2 \tensorv^*)(p,p)$$
and assume that for any $p \in T^*M$, we have $(\roughL
\tensorv^*)(p,p) \geq \rLv \|p\|_{\tensorv^*}^2$
globally on $M$ for some constant $\rLv$.
\item Finally, introduce $\RicHV$ as
$$\RicHV (Z_1,Z_2) = \frac{1}{2} \tr\left( \tensorg(Z_1, (\rnabla_{\times}^{\mathstrut} \calR)(\times,Z_2)) + \tensorg(Z_2, (\rnabla_{\times}^{\mathstrut} \calR)(\times,Z_1)) \right),$$
Assume that for any $Z \in \Gamma(TM)$,
$$\RicHV(Z,Z) \geq - 2\MRicHV \|\pr_{\calV} Z\|_{\tensorv} \| \pr_{\calH} Z\|_{\tensorh}$$
holds pointwise on $M$ for some number $\MRicHV$.
\end{enumerate}
Note that $\McalR, \MII$ and $ \MRicHV$ are always non-negative, while
this is not necessarily true for $\rRicH$ and $\rLv$. We will define
one more constant, which will always exist. For any $\alpha \in
\Gamma(T^*M)$, define $\mcalR$ as the maximal number satisfying
$$\|\alpha(\calR(\newbullet, \newbullet))\|_{\wedge^2 \tensorh^*} \geq \mcalR \|\alpha \|_{\tensorv^*}.$$
If $\rank \calV = \nu$, then 
$$\nu \mcalR^2 \leq \|\calR\|^2_{\wedge^2 \tensorg^* \otimes \tensorg} \leq \frac{n}{2} \McalR^2  = \frac{n}{2},$$ 
so the maximal value of $\mcalR$ is $(\frac{n}{2\nu})^{1/2}$ when the vertical metric has been normalized. Moreover, it can only be nonzero if $\calH$ is step $2$ equiregular as defined in
Remark~\ref{re:regularH}.

With these assumptions in place, we have the following version of a
generalized curvature-dimension inequality.
\begin{theorem} \label{th:CD} Define $\sfGamma^{\tensors^*}_2$ with
  respect to $L = \roughL$. Then $L$ satisfies \eqref{CDstar} with
  \begin{equation} \label{rhoSR} \left\{ \begin{aligned}
  	n & = \rank \calH \\
  	\rho_1  & =  \rRicH - c^{-1}, \\
        \rho_{2,0} & =  \frac{1}{2} m_{\calR}^2 - c(\MRicHV + \MII)^2, \\
        \rho_{2,1} & = \frac{1}{2} \rLv - \MII^2, \end{aligned}
    \right.\end{equation} for any positive $c > 0$.
\end{theorem}
Note that we include the possibility $c = \infty$ when $\MRicHV = \MII
= 0$. The proof is found in Section \ref{sec:proofThCD}.

We can give the following interpretation of the different terms in the
inequality.
\begin{enumerate}[\rm (i)]
\item $\McalR$ and $\mcalR$ measures how well $\tensorv$ can be
  controlled by the curvature $\calR$ of $\calH$.  To be more precise,
  for any $x \in M$, $p \in T_x^*M$, define $C_1(x)$ and $C_2(x)$ such
  that
$$C_1(x) \|p \|_{\tensorv^*} \leq \| p \circ \calR\|_{\wedge^2 \tensorh^*} \leq C_2(x) \|p \|_{\tensorv^*}$$
with $C_1(x)$ maximal and $C_2(x)$ minimal at every point. 
Then $$\frac{n}{2} \McalR \geq \sup_M \, C_2(x),$$ while $\mcalR = \inf_M C_1(x).$
\item$\Ric_{\calH}$ is a generalization of ``the Ricci
  curvature downstairs'' on sub-Riemann\-ian structures on submersions
  by Proposition~\ref{prop:RicH}~(c).
\item Both $\MII$ and $\rLv$ measure how $\tensorv$ changes
  in horizontal directions. In particular, $\rnabla \tensorv^*$ is the
  second fundamental form by Lemma~\ref{lemma:basicsrnabla}.
\item $\RicHV$ measures how ``optimal'' our subbundle $\calH$ is
  with respect to our chosen complement $\calV$ in the sense that on
  invariant sub-Riemannian structures on principal bundles, $\RicHV$
  measures how far $\calH$ is from being a Yang-Mills connection, see
  Example~\ref{ex:principal}. We will see how $\RicHV$ can be
  interpreted in a similar way in the general case in Appendix~A.4,
  Part~II.
\end{enumerate}
For further geometric interpretation, see Part~II, Section~5.2.
\begin{remark}
  In the proof of Theorem~\ref{th:CD}, we prove a curvature-dimension
  inequality without normalizing $\McalR$ in \eqref{CDnonNorm}. The
  reason why we are free to normalize $\tensorv$ such that $\McalR =1$ is the
  following. Since $\rnabla_A Z$ is independent of $\tensorv$ when
  either $A$ or $Z$ are horizontal, the bounds introduced in (i)--(v)
  behave well under scaling in the sense that for any $\ve > 0$, if we
  define the bounds relative to $\tensorv_2 = \frac{1}{\ve} \tensorv$
  rather than $\tensorv$, we will get the same inequality back for
  $\sfGamma^{\tensorh^*}_2(f) + \frac{\ell}{\ve}
  \sfGamma^{\tensorv_2^*}_2(f) = \sfGamma^{\tensorh^*}_2(f) + \ell
  \sfGamma^{\tensorv^*}_2(f)$.
\end{remark}

\begin{remark} \label{re:Baudoin}
In parallel with the development of our paper, Theorem~\ref{th:CD} for the case $\rnabla \tensorv^* =0$, 
$\MRicHV =0$ appeared in \cite{BKW14}.
\end{remark}

\subsection{Totally geodesic  foliations} \label{sec:NablaMetric0} 
Let $(M, \calH, \tensorh)$ be a
sub-Riemannian manifold with an integrable, metric-preserving
complement $\calV$. Let $\tensorv$ be a chosen metric on $\calV$ and
assume that $\rnabla \tensorv^* = 0$. By Section~\ref{sec:CompVSvol}, if $\vol$ is the
volume form of the Riemannian metric $\tensorg$ corresponding
to~$\tensorv$, then~$\roughL$ coincides with the sub-Laplacian~$\srL$
defined relative to $\vol$. By Theorem~\ref{th:CD} we also obtain a
somewhat simpler curvature-dimension inequality
\begin{equation} \label{CD} \tag{\sf CD} \sfGamma^{\tensorh^* + \ell
    \tensorv^*}_2(f) \geq \frac{1}{n} (Lf)^2 + \left(\rho_1 -
    \frac{1}{\ell} \right) \sfGamma^{\tensorh^*}(f) + \rho_2
  \sfGamma^{\tensorv^*}(f),\end{equation}
\begin{align} \label{CDII0} n = \rank \calH, \quad \rho_1 = \rRicH - c^{-1}, \quad \rho_2=
  \tfrac{1}{2} \mcalR^2 - c \MRicHV^2,\end{align}
 where $c > 0$ is arbitrary. The inequality \eqref{CD} with the additional assumption $\rho_2 >0$ was originally suggested as a generalization of the curvature-dimension inequality by Baudoin and Garofalo~\cite{BaGa12}.
  
We will also need the following relation, which is closely related to the inequality \eqref{CD}. The proof is left to
Section~\ref{sec:ProofDoubleGamma}. This result is essential for
proving the result of Theorem~\ref{th:Cond}~(b).
\begin{proposition} \label{prop:DoubleGamma} For any $f \in
  C^\infty(M)$, and any $c >0$ and $\ell >0$, we have
  \begin{align*}
    \frac{1}{4} \sfGamma^{\tensorh^*}(\sfGamma^{\tensorh^*}(f)) &\leq  \sfGamma^{\tensorh^*}(f) \left( \sfGamma^{\tensorh^* + \ell \tensorv^*}_2(f) - (\varrho_1 -\ell^{-1}) \sfGamma^{\tensorh^*}(f) - \varrho_2 \sfGamma^{\tensorv^*}(f) \right), \\
    \frac{1}{4} \sfGamma^{\tensorh^*}(\sfGamma^{\tensorv^*}(f)) &\leq
    \sfGamma^{\tensorv^*}(f) \sfGamma^{\tensorv^*}_2(f),
  \end{align*}
  where $\varrho_1 = \rRicH - c^{-1}$ and $\varrho_2 = - c \MRicHV^2.$
\end{proposition}

\begin{remark}
To give some context for Proposition~\ref{prop:DoubleGamma}, consider the following special case. Let $\tensorh = \tensorg$ be a complete Riemannian metric on $M$ with lower Ricci bound $\rho$ an choose $\tensorv^* = 0$. Inserting this in \eqref{CD} with $\ell = \infty$ gives us $\sfGamma^{\tensorg^*}_2(f) \geq \frac{1}{n} (\Delta f) + \rho \sfGamma^{\tensorg^*}(f)$ where $\Delta$ is the Laplacian of $\tensorg$. If we let $P_t = e^{t\Delta/2}$ be the heat semigroup of $\frac{1}{2} \Delta$, then the previously mentioned inequality implies the inequality $\sfGamma^{\tensorg^*}(P_t f) \leq e^{-\rho t} P_t\sfGamma^{\tensorg^*}(f)$ for any smooth, compactly supported function~$f$. However, Proposition~\ref{prop:DoubleGamma} gives us $\sfGamma^{\tensorg^*}(\sfGamma^{\tensorg^*}(f)) \leq 4\sfGamma^{\tensorg^*}(f) \big(\sfGamma^{\tensorg^*}_2(f) - \rho\sfGamma^{\tensorg^*}(f) \big)$ which imply the stronger result $\sfGamma^{\tensorg^*}(P_t f)^{1/2} \leq e^{-\rho/2 t} P_t (\sfGamma^{\tensorg^*}(f)^{1/2})$ for any smooth, compactly supported function $f$, see e.g. \cite[Section~2]{BaLe96}.
\end{remark}

\begin{remark} If a metric $\tensorv$ on $\calV$ exist with $\rnabla \tensorv = 0$,
then it is uniquely determined by its value at one point. To see this,
let $\tensorv^{\prime}$ be an arbitrary metric on $\calV$ and let
$\gamma$ be a horizontal curve in $M$. Define ${\rnabla}^{\prime}$
with respect to $\tensorv^{\prime}$. By
Lemma~\ref{lemma:basicsrnabla}~(a), we still have ${\rnabla}_{\dot
  \gamma}^{\prime} \tensorv =0$. Since $\calH$ is bracket-generating, 
the value of $\tensorv$ at any point can be determined by
parallel transport along a horizontal curve from one given point.
\end{remark}

\subsection{A convenient choice of bases for $\calH$ and
  $\calV$} \label{sec:normal} Let $(M, \calH, \tensorh)$ be a sub-Riemannian manifold with an integrable metric-preserving complement of $\calV$. Let $\tensorv$ be a metric tensor on $\calV$. To simplify the proof of
Theorem~\ref{th:CD}, we first want to introduce a convenient choice of
bases for $\calH$ and $\calV$ that will simplify our calculations,
similar to choosing the coordinate vector fields of a normal
coordinate system in Riemannian geometry. Let $\znabla$ be defined as
in \eqref{znabla}.
\begin{lemma} \label{lemma:normal} Given an arbitrary point $x_0$ of
  $M$, there are local orthonormal bases $A_1,\dots, A_n$ and $V_1,
  \dots, V_\nu$ of respectively $\calH$ and $\calV$ defined in a neighborhood around $x_0$ such that for any vector field $Z$,
  \begin{equation} \label{normal} \znabla_{Z} A_i|_{x_0} = \znabla_{Z}
    V_s|_{x_0} =0.\end{equation} In particular, these bases have the
  properties
  \begin{eqnarray} \label{bracketXi} \pr_{\calH} [ A_{i_1},
    A_{i_2}]|_{x_0} = 0, & \quad & \pr_{\calV} [ V_{s_1}, V_{s_2}
    ]|_{x_0} = 0.
  \end{eqnarray}
\end{lemma}
\begin{proof}
Define a Riemannian metric $\tensorg$ by $\tensorg = \pr_{\calH}^* \tensorh + \pr_{\calV}^* \tensorv$.  Let $(x_1,\dots, x_{n+\nu})$ be a normal coordinate system relative to $\tensorg$ centered
  at $x_0$ such that
$$\calH_{x_0} = \spn\left\{ \left. \frac{\partial}{\partial x_1} \right|_{x_0}\!\!\!, \dots, \left. \frac{\partial}{\partial x_n} \right|_{x_0}  \right\}, \quad
\calV_{x_0} = \spn\left\{ \left. \frac{\partial}{\partial x_{n+1}}
  \right|_{x_0}\!\!\!, \dots, \left. \frac{\partial}{\partial
      x_{n+\nu}} \right|_{x_0} \right\}.$$ Define $Y_j = \pr_{\calH}
\frac{\partial}{\partial x_j}$ and $Z_j = \pr_{\calV}
\frac{\partial}{\partial x_{n +j}}$. These vector fields are linearly
independent close to $x_0$. Write
$$Y_j = \sum_{i=1}^{n+\nu} a_{ij} \frac{\partial}{\partial x_i}, \qquad Z_s = \sum_{i=1}^{n+\nu} b_{is} \frac{\partial}{\partial x_i} ,$$
where
$$a_{ij}(x_0) = \left\{\begin{array}{ll} 1 & \text{ if } i=j \\ 0 & \text{ if } i \neq j \end{array} \right. , \qquad b_{is}(x_0) = \left\{\begin{array}{ll} 1 & \text{ if } i=s +n \\ 0 & \text{ if } i \neq s +n \end{array} \right. , $$
and consider the matrix-valued functions
$$a = (a_{ij})_{i,j=1}^n, \qquad b =( b_{n+r,n+s} )_{r,s=1}^\nu.$$
These matrices remain invertible in a neighborhood of $x_0$. On this
mentioned neighborhood, let $\alpha = (\alpha_{ij}) = a^{-1} $ and $\beta =
(\beta_{rs}) = b^{-1} $. Define $\widetilde Y_j = \sum_{i=1}^n \alpha_{ij}
Y_{i}$ and $\widetilde Z_s = \sum_{r=1}^\nu \beta_{rs} Z_{r}$. These bases
can then we written in the form
\begin{align*}
  \widetilde Y_j = \frac{\partial}{\partial x_j} +
  \sum_{i=n+1}^{n+\nu} \widetilde a_{ij} \frac{\partial}{\partial x_i}
  , \quad \widetilde Z_j = \frac{\partial}{\partial x_{n+j}} +
  \sum_{i=1}^{n} \widetilde b_{ij} \frac{\partial}{\partial x_i} \, ,
\end{align*}
for some functions $\widetilde a_{ij}$ and $\widetilde b_{ij}$ which
vanish at $x_0$. These bases clearly satisfy \eqref{normal} and
\eqref{bracketXi}.

Since $\znabla$ preserves the metric, we can use the Gram-Schmidt
process to obtain $A_1, \dots, A_n$ and $V_1, \dots V_{\nu}$ from
respectively $\widetilde Y_1, \dots, \widetilde Y_\nu$ and $\widetilde
Z_1, \dots, \widetilde Z_\nu$.\qed
\end{proof}

By computing $\rnabla - \znabla$, we obtain the following corollary.
\begin{corollary} \label{cor:normal} Given an arbitrary point $x_0$ of
  $M$, then around $x_0$ there are local orthonormal bases $A_1,\dots,
  A_n$ and $V_1, \dots, V_\nu$ of respectively $\calH$ and $\calV$
  such that for any vector field $Z$,
  \begin{align*}
    &\rnabla_{Z} A_i|_{x_0}  = \frac{1}{2} \sharp \tensorg(Z, \calR(A_i, \newbullet))|_{x_0}, \\
    &\rnabla_{Z} V_s|_{x_0} = - \frac{1}{2} \sharp (\rnabla_{Y}
    \tensorg)(V_s, \newbullet) |_{x_0} ,
  \end{align*}
  where $\sharp: T^*M \to TM$ is the identification defined relative
  to $\tensorg$. In particular, these bases have the properties $\pr_{\calH} [ A_{i_1}, A_{i_2}]|_{x_0} = 0,$ and $\pr_{\calV} [V_{s_1}, V_{s_2} ]|_{x_0} = 0$.
\end{corollary}

\subsection{Proof of Theorem~\ref{th:CD}} \label{sec:proofThCD} Let
$\tensorh$ and $\tensorv$ be the respective metrics on $\calH$ and
$\calV$ that give us a Riemannian metric $\tensorg = \pr^*_{\calH}
\tensorh + \pr_{\calV}^* \tensorv$. Let $\flat: TM \to T^*M$ be the
map $v \mapsto \tensorg(v, \newbullet)$ with inverse $\sharp$. Let
$\shv$ be defined similar to the definition of~$\shh$ in
Section~\ref{sec:srmanifolds}. Note that $\shh = \pr_{\calH} \sharp$
and $\shv = \pr_{\calV} \sharp$.

Let $A_1, \dots, A_n$ be as in Corollary~\ref{cor:normal} relative to
some point $x_0$. Clearly, for any $f \in C^\infty(M)$, we have $L
f(x_0) = \sum_{i=1}^n A_i^2 f(x_0).$ Note also that
$$\rnabla_A df(Z) = \rnabla_Z df(A) - df(T^{\rnabla}(A,Z)) = \rnabla_Z df(A) + df(\calR(A,Z)).$$
In the following calculations, since $\calV$ is metric-pre\-serv\-ing, keep in mind that
$$\rnabla_A \shh df =\shh \rnabla_A df,$$
while $\rnabla_{A} \shv df = \shv \rnabla_{A} df +
(\rnabla_{A} \tensorv^*)(df, \newbullet)$,.

Below, all terms are evaluated at $x_0$. We first note that for any
$\ell > 0$,
\begin{align*}
  \sfGamma^{\tensorh^* + \ell \tensorv^*}_2(f) &= \frac{1}{2}\sum_{i=1}^n A_i^2 \left(\|df\|^2_{\tensorh^*} + \ell \|df\|^2_{\tensorv^*} \right) - \tensorh^*(df, dLf) - \ell \tensorv^*(df, dLf) \\
  &=  \sum_{i=1}^n A_i \rnabla_{A_i} df( \shh df) + \ell \sum_{i=1}^n A_i \rnabla_{A_i} df( \shv df)\\
  &\quad + \frac{1}{2} \ell\sum_{i=1}^n A_i (\rnabla_{A_i}
  \tensorv^*)(df, df)
  - (\shh df + \ell \shv df)\left( \sum_{i=1}^n \rnabla_{A_i} df(A_i) \right) \\
  &=  \sum_{i=1}^n \rnabla_{A_i} \rnabla_{\shh df} df(A_i) + \sum_{i=1}^n A_i df(\calR(A_i, \shh df))\\
  &\quad + \ell \sum_{i=1}^n \rnabla_{A_i} \rnabla_{\shv df} df(A_i)
  - \sum_{i=1}^n \rnabla_{\shh df} \rnabla_{A_i} df(A_i)\\&\quad - \ell \sum_{i=1}^n \rnabla_{\shv df} \rnabla_{A_i} df(A_i) - \frac{1}{2} \ell \sum_{i=1}^n df(\calR(A_i, \shh \rnabla_{A_i} df)) \\
  &\quad + \frac{1}{2} \ell (\roughL \tensorv^*)(df, df) + \ell
  \sum_{i=1}^n (\rnabla_{A_i} \tensorv^*)(\rnabla_{A_i} df, df).
\end{align*}
Observe that
\begin{align*}
  \rnabla_{A_i} &\rnabla_{\shh df} df(A_i) - \rnabla_{\shh df} \rnabla_{A_i} df(A_i) \\
  &=  \sum_{i=1}^n \tensorg(R^{\rnabla}(A_i, \shh df) \shh df, A_i) + \sum_{i=1}^n \rnabla_{[A_i, \shh df]} df(A_i) \\
  &=  \Ric_{\calH}(\shh df, \shh df) + \sum_{i=1}^n \rnabla_{\rnabla_{A_i} \shh df } df(A_i) + \sum_{i=1}^n \rnabla_{\calR(A_i, \shh df)} df(A_i) \\
  &=  \Ric_{\calH}(\shh df, \shh df) + \sum_{i=1}^n \| \rnabla_{A_i} df\|^2_{\tensorh^*} \\
  &\quad - \sum_{i=1}^n df(\calR(A_i, \shh \rnabla_{A_i} df)) +
  \sum_{i=1}^n \rnabla_{A_i} df(\calR(A_i, \shh df)),
\end{align*}
while
\begin{align} \label{VerticalComputation} \rnabla_{A_i}& \rnabla_{\shv
    df} df(A_i) - \rnabla_{\shv df} \rnabla_{A_i} df(A_i) \\ \nonumber
  &= \sum_{i=1}^n \tensorg(R^{\rnabla}(A_i, \shv df) \shh df, A_i) \\
  \nonumber & \quad+ \sum_{i=1}^n \rnabla_{\rnabla_{A_i} \shv df}
  df(A_i) - \sum_{i=1}^n \rnabla_{\rnabla_{\shv df} \, A_i} df(A_i) \\
  \nonumber & = \sum_{i=1}^n \rnabla_{\shv \rnabla_{A_i} df +
    (\rnabla_{A_i} \tensorv^*)(df, \newbullet)} df(A_i) - \frac{1}{2}
  \sum_{i=1}^n \rnabla_{\shh df(\calR(A_i, \newbullet))} df(A_i) \\
  \nonumber &= \sum_{i=1}^n \| \rnabla_{A_i} df\|^2_{\tensorv^*} +
  \sum_{i=1}^n (\rnabla_{A_i} \tensorv^*)(df, \rnabla_{A_i} df) \\
  \nonumber &\quad - \frac{1}{2} \sum_{i=1}^n df(\calR(A_i, \shh
  \rnabla_{A_i} df)) + \| df(\calR(\newbullet,
  \newbullet))\|^2_{\wedge^2 \tensorg^*}.
\end{align}
Hence
\begin{align*}
  \sfGamma^{\tensorh^* + \ell \tensorv^*}_2(f) = &  \sum_{i=1}^n \| \rnabla_{A_i} df\|^2_{\tensorh^*} + \Ric_{\calH}(\shh df, \shh df) - \sum_{i=1}^n df(\calR(A_i, \shh \rnabla_{A_i} df))  \\
  & + \sum_{i=1}^n \rnabla_{A_i} df(\calR(A_i, \shh df))  + \sum_{i=1}^n A_i df(\calR(A_i, \shh df)) \\
  & + \ell \sum_{i=1}^n \| \rnabla_{A_i} df\|^2_{\tensorv^*} + \ell
  \sum_{i=1}^n (\rnabla_{A_i} \tensorv^*)(df, \rnabla_{A_i} df) \\
  \nonumber
  & - \ell \sum_{i=1}^n df(\calR(A_i, \shh \rnabla_{A_i} df)) + \ell \| df(\calR(\newbullet, \newbullet))\|^2_{\wedge^2 \tensorg^*} \\
  & + \frac{1}{2} \ell (\roughL \tensorv^*)(df, df) + \ell
  \sum_{i=1}^n (\rnabla_{A_i} \tensorv^*)(\rnabla_{A_i} df, df).
\end{align*}

By realizing that
\begin{align*} \sum_{i=1}^n A_i df(\calR(A_i, \shh df)) = &
  \RicHV(\sharp df, \sharp df) + \sum_{i=1}^n \rnabla_{A_i}
  df(\calR(A_i, \shh df)) \\ & + \sum_{i=1}^n df(\calR(A_i, \shh
  \rnabla_{A_i} df)),\end{align*} and that
$$df(\calR(A_i, \shh \rnabla_{A_i} df)) = \|df(\calR(\newbullet, \newbullet))\|^2_{\wedge^2 \tensorg^*}, $$
we obtain
\begin{align} \label{Gamma2expression} \sfGamma^{\tensorh^* + \ell
    \tensorv^*}_2(f) = & \sum_{i=1}^n \| \rnabla_{A_i}
  df\|_{\tensorh^*}^2 +\RicH(\shh df, \shh df) \\ \nonumber & +
  \RicHV(\sharp df, \sharp df) + 2 \sum_{i=1}^n \rnabla_{A_i}
  df(\calR(A_i, \shh df)) \\ \nonumber & + \ell \sum_{i=1}^n
  \|\rnabla_{A_i} df\|_{\tensorv^*}^2 + 2 \ell \sum_{i=1}^n
  (\rnabla_{A_i} \tensorv^*)(df, \rnabla_{A_i} df) \\ \nonumber & +
  \frac{1}{2} \ell (\roughL \tensorv^*)(df, df) .
\end{align}
Clearly
\begin{align*}
  &   \ell \sum_{i=1}^n \|\rnabla_{A_i} df\|_{\tensorv^*}^2+ 2 \rnabla_{A_i} df(\calR(A_i, \shh df))  + 2 \ell \sum_{i=1}^n (\rnabla_{A_i} \tensorv^*)(df, \rnabla_{A_i} df) \\
  &\quad\geq  - \frac{1}{\ell} \sum_{i=1}^n \| \ell (\rnabla_{A_i} \tensorv^*)(df, \newbullet) + \calR(A_i, \shh df)\|_{\tensorv}^2 \\
  &\quad \geq - \frac{\McalR^2}{\ell} \sfGamma^{\tensorh^*}(f) - 2
  \McalR \MII \sqrt{\sfGamma^{\tensorh^*}(f)
    \sfGamma^{\tensorv^*}(f)}- \ell \MII^2 \sfGamma^{\tensorv^*}(f),
\end{align*}
and also
\begin{align*}
  & \sum_{i=1}^n \|\rnabla_{A_i} df\|^2_{\tensorh^*} = \sum_{i,j=1}^n \left( \frac{1}{2} (\rnabla_{A_i,A_j}^2 f + \rnabla^2_{A_j, A_i} f) + \frac{1}{2} (\rnabla_{A_i,A_j}f - \rnabla_{A_i, A_j} f )\right)^2 \\
  &\quad=   \sum_{i,j=1}^n \left( \frac{1}{2} (\rnabla_{A_i,A_j}^2 f + \rnabla^2_{A_j,A_i} f) \right)^2 + \sum_{i,j=1}^n \left(  \frac{1}{2} (\rnabla_{A_i,A_j}f - \rnabla_{A_i,A_j} f )\right)^2 \\
  &\quad\geq \sum_{i=1}^n (\rnabla_{A_i,A_i}^2 f)^2 +
  \sum_{i,j=1}^n(\frac{1}{2} df(\calR(A_i,A_j)))^2 \geq \frac{1}{n} (L
  f)^2 + \frac{1}{2} \mcalR^2 \sfGamma^{\tensorv^*}(f).
\end{align*}
In conclusion
\begin{align} \label{CDnonNorm} \sfGamma^{\tensorh^* + \ell
    \tensorv^*}_2(f) & \geq \frac{1}{n} (L f)^2 + \left(\rRicH -
    \frac{\McalR^2}{\ell} \right) \sfGamma^{\tensorh^*}(f) \\
  \nonumber &\quad - (2 \MRicHV + 2 \McalR \MII)
  \sqrt{\sfGamma^{\tensorh^*}(f) \sfGamma^{\tensorv^*}(f)} \\
  \nonumber &\quad + \frac{1}{2} (\mcalR^2 + \ell( \rLv - 2\MII^2))
  \sfGamma^{\tensorv^*}(f) \\ \nonumber &\geq  \frac{1}{n} (L f)^2 +
  \left( \rRicH - \frac{1}{c} - \frac{\McalR^2}{\ell} \right)
  \sfGamma^{\tensorh^*}(f) \\ \nonumber &\quad+ \frac{1}{2} \left( \mcalR^2
    - 2c (\MRicHV + \McalR \MII)^2\right) \sfGamma^{\tensorv^*}(f) \\
  \nonumber &\quad + \ell (\rLv -2 \MII^2) \sfGamma^{\tensorv^*}(f) .
\end{align}

\subsection{Proof of
  Proposition~\ref{prop:DoubleGamma}} \label{sec:ProofDoubleGamma} Let
$A_1, \dots, A_n$ be a local orthonormal basis of~$\calH$. From the
assumption $\rnabla \tensorv^* = 0$ and \eqref{Gamma2expression}, we
obtain $\sfGamma^{\tensorv^*}_2(f) = \sum_{i=1}^n \|\rnabla_{A_i}
df\|^2_{\tensorv^*}$ and from this, we know
\begin{align*}
  \sfGamma^{\tensorh^*}(\sfGamma^{\tensorv^*}(f)) &= 4 \sum_{i=1}^n \tensorv^*(\rnabla_{A_i} df, df)^2 \\
  &\leq 4 \sum_{i=1}^n \|\rnabla_{A_i} df\|^2_{\tensorv^*}
  \|df\|_{\tensorv^*} = 4 \sfGamma^{\tensorv^*}_2(f)
  \sfGamma^{\tensorv^*}(f).
\end{align*}
Similarly, from \eqref{Gamma2expression},
\begin{align*}
  \sum_{i=1}^n \| \rnabla_{A_i} df\|^2_{\tensorh^*} 
&=\sfGamma^{\tensorh^* + \ell \tensorv^*}_2(f) - \Ric_{\calH}(\shh df,\shh df) - 2 \RicHV(\shv df,\shh df)\\ &\quad- 2 \sum_{i=1}^n\rnabla_{A_i} df(\calR(A_i, \shh df)) 
+ \ell \sum_{i=1}^n \|\rnabla_{A_i} df\|^2_{\tensorv^*} \\
  &\leq \sfGamma^{\tensorh+ \ell \tensorv^*}_2(f) - (\rRicH - c^{-1}-
  \ell^{-1}) \sfGamma^{\tensorh^*}(f) + c \MRicHV^2
  \sfGamma^{\tensorv^*}(f),
\end{align*}
and the result follows.\qed

\subsection{For a general choice of $L$} \label{sec:GeneralL} Let $(M,
\calH, \tensorh)$ and $\calV$ be as in Section~\ref{sec:GenCD} and let
$\roughL$ be the sub-Laplacian of $\calV$.  For a general choice of
$L$ with $\tensorq_L = \tensorh^*$, write $L = \roughL + Z$ for some
vector field $Z$. We want to use our generalized curvature-dimension
inequality for $\roughL$ to extend it to a more general class of
operators. Unfortunately, our possibilities are somewhat limited.

\begin{proposition} \label{prop:LGeneral} Let $L = \roughL + Z$ where
  $Z \in \Gamma(\calV)$ is a non-zero vertical vector field. Then $L$
  satisfies \eqref{CDstar} with $\rank \calH < n \leq \infty$ and
  $\rho_1, \rho_{2,0}, \rho_{2,1}$ given by
  \begin{align*}  \rho_1  & =  \rRicH - c^{-1}, \\
    \rho_{2,0} & =  \frac{1}{2} {m_{\calR}^2} - c(\scrM_{\calH\calV}^Z + \MII)^2 - \frac{1}{n - \rank \calH} \|Z\|_{\tensorv^*}^2, \\
    \rho_{2,1} & = \frac{1}{2} \rLv - \MII^2 - \scrN^2,
  \end{align*}
  for any positive $c > 0$. Here, $-\scrM_{\calH\calV}^Z$ is a lower bound
  of
  \begin{align*} \Ric_{\calH\! \calV}^Z(A_1,A_2)& : =  \RicHV(A_1,A_2) + \frac{1}{2} \tensorh(\pr_{\calH} A_1, \rnabla_{\pr_{\calH} A_2} Z) + \frac{1}{2} \tensorh(\pr_{\calH} A_2, \rnabla_{\pr_{\calH} A_1} Z) \\
    &\quad + \frac{1}{2} \tensorv( \pr_{\calV} A_1, \calR(Z,A_2)) +
    \frac{1}{2} \tensorv( \pr_{\calV} A_2, \calR(Z,A_1)) \end{align*} and
  $\scrN$ is a lower bound of $\tensorv(\pr_{\calV} \newbullet,
  \rnabla_{\pr_{\calV\newbullet}} Z)$. The other constants are as in
  Section~\textup{\ref{sec:GenCD}}.
\end{proposition}

\begin{proof} Let $\shh$ be defined as in
  Section~\ref{sec:srmanifolds} and let $\shv$ be defined
  analogously. Then the result follows from the identities,
  \begin{align*}
    \frac{1}{2} Z \sfGamma^{\tensorh^*}(f) - \sfGamma^{\tensorh^*}(Zf,f) & =   df(\rnabla_{\shh df} Z) + df(\calR(Z, \shh df)), \\
    \frac{1}{2} Z \sfGamma^{\tensorv^*}(f) - \sfGamma^{\tensorv^*}(Zf,f) &=   df(\rnabla_{\shv df} Z) + \frac{1}{2} (\rnabla_Z \tensorv^*)(df, df), \\
    \frac{1}{\rank \calH}(\roughL f)^2 & \geq \frac{1}{n}(\roughL f +
    Zf)^2 - \frac{1}{n-\rank \calH} (Zf)^2 
  \end{align*}
  which hold for any vector field $Z$ (not necessarily vertical).\qed
\end{proof}

%
%
%

The proof of Proposition~\ref{prop:LGeneral} also shows why it is
complicated to extend this formalism to the more general case. If
$\pr_{\calH} Z \neq 0$, then the term $\ell df( \rnabla_{\shh df}
\pr_{\calH} Z)$ requires a lower bound on the form $\ell b
\sfGamma^{\tensorh^*}(f) + \ell b \sfGamma^{\tensorv^*}(f)$ or $b
\sfGamma^{\tensorh^*}(f) + \ell^2 b \sfGamma^{\tensorv^*}(f)$, both of
which would be outside of our formalism.

\subsection{Generalization to the case when $\calV$ is not integrable} \label{sec:NotIntegrable}
Not every vector bundle has an integrable complement \cite{Bot70}, not to mention 
a metric-pre\-serv\-ing one. We give a brief comment on how our
results can be generalized to the case when $\calV$ is not integrable.

Let $\calR$ be defined as in \eqref{EhresmannCurv} and let
$\overline{\calR}$ be defined by
$$\overline{\calR}(A,Z) := \pr_{\calH} [ \pr_{\calV} A, \pr_{\calV} Z], \quad A,Z \in \Gamma(TM).$$
We will adopt the terminology of \cite[Ch II.8]{KMS93} and call
$\calR$ and $\overline{\calR}$ respectively the curvature and the
co-curvature of $\calH$. Then our theory can still be applied with the
following modifications.
\begin{enumerate}[(a)]
\item We consider a complement $\calV$ as metric-preserving if
  \begin{equation} \label{mp} \pr_{\calH}^* \calL_{V} (\pr_{\calH}^*
    \tensorh) = 0, \quad V \in \Gamma(\calV).\end{equation} Notice
  the difference between the formula above and \eqref{mpintegral}. In
  fact, \eqref{mpintegral} holds if and only if \eqref{mp} holds and
  $\calV$ is integrable. Note that \eqref{mp} is equivalent to stating
  that $\rnabla \tensorh^* = 0$ with respect to any connection
  $\rnabla$ defined as in \eqref{rnabla} using some metric $\tensorg$
  which tames $\tensorh$ and makes $\calV$ the orthogonal complement
  of $\calH$.
\item In Section~\ref{sec:CurvatureDimension}, we now have that
  $T^{\rnabla} = - \calR - \overline{\calR}$ in
  Lemma~\ref{lemma:basicsrnabla}~(c). As a consequence, in
  Lemma~\ref{lemma:curvaturernabla}~(c), we obtain
$$\tensorg(R^{\rnabla}(A, Y) Z - R^{\rnabla}(A, Z) Y, A) = \tensorg( \overline{\calR}(Y, \calR(Z,A)) - \overline{\calR}(Z, \calR(Y,A)), A).$$
Hence, $$\tensorg(R^{\rnabla}(A, \pr_{\calH} Y) \pr_{\calH} Z, A)
=\tensorg( R^{\rnabla}(A, \pr_{\calH} Z) \pr_{\calH} Y, A),$$ however, we now have
$\tensorg(R^{\rnabla}(A, \pr_{\calV} Y) Z,A)  = \tensorg( \overline{\calR}(Y, \calR(Z,A)), A)$.
\item In Section~\ref{sec:normal}, Lemma~\ref{lemma:normal} still
  holds, but since $\rnabla - \znabla$ is different, in
  Corollary~\ref{lemma:normal} we have
$$\rnabla_Z V_s |_{x_0} = - \frac{1}{2} \sharp (\rnabla_Z \tensorg)(V_s, \newbullet )|_{x_0} + \frac{1}{2} \sharp \tensorg(Z, \overline{\calR}(V_s, \newbullet))|_{x_0}.$$
\item In the proof of Theorem~\ref{sec:proofThCD} in
  Section~\ref{sec:proofThCD}, the only difference is that in Eq.~\eqref{VerticalComputation}, we cannot be sure that the term
  $$\sum_{i=1}^n \tensorg(R^{\rnabla}(A_i, \shv df) \shh df, A_i) = \tr
  \overline{\calR}(\shv df, \calR(\shh df, \newbullet))$$
  vanishes. Hence, we require the separate assumption that for any
  vector $v \in TM$ on $M$, we have
  \begin{equation} \label{TrEq0} \tr \overline{\calR}(v, \calR(v,
    \newbullet)) = 0.\end{equation}
  The same assumption also guarantees that
  $$\tr R^{\rnabla}(\pr_{\calH} \newbullet, Z_1) Z_2 = \tr R^{\rnabla}(\pr_{\calH} \newbullet, \pr_{\calH} Z_1) \pr_{\calH} Z_2,$$
  so the definition of $\RicH$ in Proposition~\ref{prop:RicH} is still valid. If this does indeed hold, then
  Theorem~\ref{th:CD} remains true even if $\calV$ is not
  integrable. As a consequence, all further results in this paper and
  in Part~II also hold in this case. Since we do not have any
  geometric interpretation for the requirement \eqref{TrEq0}, we
  prefer to mainly consider the case when $\calV$ is integrable.

  The only exception are the results in Section~\ref{sec:GeneralL}. Even if
  \eqref{TrEq0} is satisfied, these results do
  do not hold when $\calV$ is not integrable.
\end{enumerate}

\begin{example}
  Consider the Lie algebra $\mathfrak{su}(2)$ with basis $A, B, C$
  satisfying commutation relations
$$[A,B] = C, \quad [A, C] = -B, \quad [B, C] = C.$$
Consider its complexification, which is isomorphic to
$\mathfrak{sl}(2, \comp)$. Define a sub-Rie\-mann\-ian manifold
$(\mathrm{SL}(2,\comp), \calH, \tensorh)$ by considering $i A, iB, iC$
and $C$ as an orthonormal basis for $\calH$. Here, we have used the
same symbol for an element of the Lie algebra and its corresponding
left invariant vector field. Then $\calV$ spanned by $A$ and $B$ is a
metric-preserving complement that is not integrable, but satisfies
\eqref{TrEq0}.
\end{example}

\section{Spectral gap and examples} \label{sec:Examples}
\subsection{The curvature-dimension inequality and a bound for the spectral gap}
Let $(M, \calH, \tensorh)$ be a compact sub-Riemannian manifold where $\calH$
is bracket-gen\-er\-at\-ing. Let $L$ be a smooth second order operator
without constant term satisfying $\tensorq_L = \tensorh^*$. Assume
also that $L$ is symmetric with respect to some volume form $\vol$ on
$M$. Since the metric $\metricd_{cc}$ induced by $\tensorh^*$ is
obviously complete on $M$, we have that $L$ is essentially
self-adjoint on $C^\infty_c(M)$ by \cite[Sec 12]{Str86}. Denote its
(unique) self-adjoint extension to an operator on $L^2(M,\vol)$ also by $L$.

\begin{proposition} \label{prop:SpectralGap} Assume that $L$ satisfies
  \eqref{CDstar} with $\rho_{2,0} > 0$. Let $\lambda$ be any nonzero
  eigenvalue of $L$. Then
$$\frac{n \rho_{2,0}}{n + \rho_{2,0}(n-1)} \left(\rho_1 - \frac{ k_2}{\rho_{2,0}} \right) 
\leq - \lambda, \qquad k_2 = \max \{0, -\rho_{2,1} \}.$$
\end{proposition}
\begin{proof}
  Since $\calH$ is bracket-generating, we know that $L - \lambda$ is hypoelliptic for any $\lambda$
  by \cite{Hor67}, so all eigenfunctions of $L$ are
  smooth. If we write $\langle f,g \rangle = \int_M fg \dvol$, note
  that $\int_M L f \dvol = \langle L f, 1 \rangle = 0$ and $\int_M
  \sfGamma^{\tensorh^*}(f,g) \dvol = - \langle f, L g \rangle$ for $f,g \in
  C^\infty(M)$. Since $L$ is a nonpositive operator, any nonzero
  eigenvalue is negative. From \eqref{CDstar} we get
  \begin{align*}
    \int_M&( \sfGamma^{\tensorh^*}_2(f) + \ell \sfGamma^{\tensorv^*}_2(f) ) \dvol 
= \langle Lf, Lf \rangle - \langle \sfGamma^{\tensorv^*}(f,Lf), \ell \rangle  \\
    &\geq \frac{1}{n} \langle Lf, Lf \rangle - \left(\rho_1 -
      \frac{1}{\ell} \right) \langle f, Lf \rangle +
    \langle\sfGamma^{\tensorv^*}(f), \rho_{2,0} + \ell \rho_{2,1}
    \rangle.
  \end{align*}
  Hence, if $f$ satisfies $L f = \lambda f$, then
$$\frac{n-1}{n} \lambda^2 \langle f, f\rangle \geq - \lambda \left(\rho_1 - \frac{1}{\ell} \right) \langle f, f\rangle + \langle \sfGamma^{\tensorv^*}(f), \rho_{2,0} + \ell (\rho_{2,1} +\lambda) \rangle.$$
We choose $\ell = \frac{\rho_{2,0}}{-\lambda + k_2}$ and obtain
$$\frac{n-1}{n} \lambda^2 \geq - \lambda \left(\rho_1 - \frac{- \lambda + k_2}{\rho_{2,0} } \right) ,$$
from which the result follows.\qed
\end{proof}

Let $\tensorg$ be a Riemannian metric taming
$\tensorh$ such that the orthogonal complement $\calV$ of $\calH$ is
integrable. We consider the special case when $\rnabla \tensorg = 0$ where $\rnabla$ is
defined as in \eqref{rnabla}. Then $\calV$ is a metric-preserving
complement and $\rnabla \tensorv^* = 0$ where $\tensorg|_{\calV} =:
\tensorv$. Let $\vol$ be the volume form of $\tensorg$ and let $\srL$
the sub-Laplacian of $\vol$, which will also be the sub-Laplacian of
$\calV$.
\begin{corollary}
  Assume that the assumptions of Theorem~\upshape\ref{th:CD} hold with
  $$\kappa = \frac{1}{2} \rRicH \mcalR^2 - \MRicHV^2 >0.$$ Then, for any
  $\rank \calH \leq n \leq \infty$,
$$\left( \frac{2 \kappa}{2 \MRicHV + \mcalR \sqrt{2 \rRicH + 2 \frac{n-1}{n} \kappa}} \right)^2 \leq - \lambda.$$
\end{corollary}
\begin{proof}
  This follows from formulas \eqref{CDII0} and by choosing the
  optimal value of $c$.\qed
\end{proof}

\subsection{Privileged metrics} \label{sec:PrivMetric} Let $(M, \calH,
\tensorh)$ be a sub-Riemannian manifold with $\calH$
bracket-generating and equiregular of step $r$ as in
Remark~\ref{re:regularH}. Let $\dim M = n + \nu$ with $n$ being the
rank of $\calH$. Any Riemannian metric $\widetilde \tensorg$ on $M$
such that $\widetilde \tensorg|_{\calH} = \tensorh$, gives us
automatically a splitting $$TM = \calH \oplus \calV^1 \oplus \cdots
\oplus \calV^{r-1}$$ where $\calV^k$ is the orthogonal complement of
$\calH^k$ in $\calH^{k+1}$. Conversely, associated to each such splitting,
there exist a canonical way of constructing a Riemannian metric~$\widetilde
\tensorg$ taming $\tensorh$, which we will call {\it privileged}. Define a vector bundle
morphism
$$\Psi\colon \calH \oplus \calH^{\otimes 2} \oplus \cdots \oplus \calH^{\otimes r} 
\to \calH \oplus \calV^1 \oplus \cdots \oplus \calV^{r-1},$$
such that $\Psi$ is the identity on the first
component, while elements in $\calH^{\otimes j}$, $j\geq2$ are sent
to $\calV^{j-1}$ by
$$A_1 \otimes A_2 \otimes \dots \otimes A_j \mapsto 
\pr_{\calV_{j-1}} [A_1, [ A_2 [ \cdots [A_{j-1} , A_j]]\cdots ]].$$
Giving $\calH^{\otimes j}$ the metric $\tensorh^{\otimes j}$,
$\Psi$ induces a metric $\widetilde \tensorg$ on $TM$ by requiring that
$\Psi|_{(\ker \Phi)^\perp}$ is a fiberwise isometry.

Assume that $\calV = \oplus_{k=1}^{r-1} \calV_k$ is an integrable
metric-preserving complement. Define~$b$ as the minimal number such that $\|\calR(v,\newbullet)\|_{\widetilde \tensorg^* \otimes \widetilde \tensorg} \leq b \|\pr_{\calH} v\|_{\widetilde \tensorg}$ for any $v \in TM$.
Note that
$$\frac{2 \dim \calV_1}{n} = \frac{2}{n} \|\calR\|_{\wedge^2 \widetilde \tensorg \otimes \tensorg}^2 \leq b^2 \leq 2 \|\calR\|_{\wedge^2 \widetilde \tensorg \otimes \tensorg}^2 = 2 \dim \calV_1.$$

We normalize the vertical part of the
metric by defining
$$\tensorv = \frac{1}{b^2} \widetilde \tensorg|_{\calV} \quad  \text{and } \tensorg = \pr_{\calH}^* \tensorh + \pr_{\calV}^* \tensorv.$$
Then $\McalR =1$, while $\mcalR = \frac{1}{b}$ if $r=2$ and
$0$ otherwise. Furthermore, if $\calH$ is of step~$2$ and $\rnabla_A
\calR = 0$ for any $A \in \Gamma(\calH)$, then~$\rnabla \tensorv^* =
0$ and $\RicHV =0$. Hence, for this special case, the sub-Laplacian $\srL$ of $\calV$ or
equivalently the volume form of $\tensorg$, satisfies curvature-dimension inequality
$$\sfGamma^{\tensorh^*+ \ell \tensorv^*}_2(f) \geq \frac{1}{n} (\srL f)^2 + (\rRicH -\ell^{-1}) \sfGamma^{\tensorh^*}(f) + \frac{1}{2b^2} \sfGamma^{\tensorv^*}(f),$$
for any $\ell > 0$. As a consequence, if $M$ is compact with $\rRicH
>0$ and $\lambda$ is a non-zero eigenvalue of $\srL$ then
$$ \frac{n}{n( 2b^2 + 1) -1} \rRicH \leq - \lambda ,$$
from Proposition~\ref{prop:SpectralGap}.

The volume forms of all privileged metric taming $\tensorh$ coincide
and is called Popp's measure. For more details, see \cite{ABGR09}.

\subsection{Sub-Riemannian manifolds with transverse
  symmetries} \label{sec:TS} For two Riemann manifolds $(M_j, \calH_j,
\tensorh_j), j = 1,2$, \emph{a sub-Riemannian isometry} $\phi:M_1 \to M_2$ is
a diffeomorphism such that $\phi^* \tensorh_2^* = \tensorh_1^*$. The
later requirement can equivalently be written as $\phi_* \calH_1
\subseteq \calH_2$ and $\tensorh_2(\phi_* v, \phi_* v) =
\tensorh_1(v,v)$ for any $v \in \calH_1$. \emph{An infinitesimal isometry} of
a sub-Riemannian manifold $(M, \calH, \tensorh)$ is a vector field $V$
such that
\begin{equation} \label{InfIsometry} \calL_{V} \tensorh^* =
  0.\end{equation} If $V$ is complete with flow $\phi_t$, then this
flow is an isometry from $M$ to itself for every fixed $t$.

We will introduce \emph{sub-Riemannian manifolds with transverse
  symmetries} according to the definition found in \cite{BaGa12}. This
is a special case of a sub-Riemannian manifold with a
metric-preserving complement and consists of sub-Riemannian manifolds
$(M, \calH, \tensorh)$ with an integrable complement $\calV$ spanned
by $\nu$ linearly independent vector fields $V_1, \dots, V_\nu$, such
that each of these vector fields is an infinitesimal isometry. The
subbundle $\calV$ will then be a metric preserving complement. If
$\tensorv$ is the metric on $\calV$ defined such that $V_1, \dots,
V_\nu$ forms an orthonormal basis, then~$\rnabla \tensorv^* =
0$. Hence, a complement spanned by transverse symmetries gives us a
totally geodesic foliation.

Since we assume that $\calH$ was bracket-generating, it follows that the span
$V_1, \dots, V_\nu$ actually is a subalgebra of $\Gamma(TM)$. Indeed, since $\calH$ is
bracket-generating, any function $f \in C^\infty(M)$ satisfying $Af = 0$ for all
$A \in \Gamma(\calH)$ must be a constant. Since $[V_i, V_j] = \sum_{s=1}^\nu f_s V_s$ must also
be an infinitesimal isometry for any $i,j$, we have that for all $A \in \Gamma(\calH)$,
$$0 = \pr_{\calV} [A, [V_i, V_j]] = \sum_{s=1}^\nu (A f_s) V_s.$$
It follows that each $f_s$ is constant, and the span of $V_1, \dots, V_\nu$ is a subalgebra.
If all of the vector fields are complete, we get a corresponding group action.
We will then be in the following case.

\begin{example} \label{ex:principal} Let $G$ be a Lie group with Lie
  algebra $\frakg$. Consider a principal bundle $G \to M
  \stackrel{\pi}{\to} B$ over a Riemannian manifold $(B,
  \widecheck{\tensorg})$ with $G$ acting on the right. An Ehresmann
  connection $\calH$ on $\pi$ is called principal if $\calH_x \cdot a
  = \calH_{x \cdot a}$. For any such principal connection $\calH$,
  define $\tensorh = \pi^* \widecheck{\tensorg}|_{\calH}.$ Then $G$
  acts on $(M, \calH, \tensorh)$ by isometries and so, for each $A
  \in \frakg$, the vector field $\sigma(A)$ defined by
$$\sigma(A)|_x = \left. \frac{d}{dt} \right|_{t=0} x \cdot \exp_G(At)$$ 
is an infinitesimal isometry. This is hence a
sub-Riemannian manifold with transverse symmetries. Let $\omega: TM
\to \frakg$ be {\it principal curvature form} of $\calH$, i.e. the
$\frakg$-valued one-form defined by
$$\ker \omega = \calH, \qquad \omega(\sigma(A)) = A.$$
Then for any inner product $\langle \newbullet, \newbullet \rangle$ on
$\frakg$, define a Riemannian metric $\tensorg$ by $\tensorg(v,v) =
\widecheck{\tensorg}(\pi_* v, \pi_* v) + \langle \omega(v),
\omega(v)\rangle$, $v \in TM.$ This Riemannian metric satisfies $\rnabla
\tensorg = 0$. We assume that the vertical part of $\tensorg$ is normalized
so that $\McalR =1$.

In general, the metric $\tensorg$ is not invariant under the group
action. The latter only hold if $\langle \newbullet, \newbullet
\rangle$ is bi-invariant inner product on $\frakg$. If such an inner
product exist, it induces a metric tensor $\widecheck{\tensorv}$ on
the vector bundle $\Ad(M) \to B$, where $\Ad(M)$ is the quotient of $M
\times \frakg$ by the action $(x, A) \cdot a = (x \cdot a,
\Ad(a^{-1})A).$ Any $\frakg$-valued from $\alpha$ on $M$ that vanish on $\calV$ and satisfies
$\alpha(Z_1 \cdot a, \dots, Z_j \cdot a) = \Ad(a^{-1}) \alpha(Z_1, \dots,
Z_j)$ can be considered as an $\Ad(M)$-valued form on $B$. This
includes the {\it curvature form} $\Omega(Z_1,Z_2) = d\omega(Z_1,Z_2) +
[\omega(Z_1),\omega(Z_2)] = - \omega(\calR(Z_1,Z_2)).$

Conversely, any section $F$ of $ \Ad(M)$ can be considered a function
$F: M \to \frakg$ satisfying $F(x \cdot a) = \Ad(a^{-1}) F(x).$ Define
a connection $\nabla^\omega$ on $\Ad(M)$ by formula
$\nabla_{\check{Z}}^\omega F = dF(h\check{Z})$ for $\check{Z} \in \Gamma(TB)$ and let $d_{\nabla^\omega}$ be the corresponding
covariant exterior derivative of $\Ad(M)$-valued forms on $B$. If we consider $\Omega$
as a $\Ad(M)$-valued 2-form and $\delta_{\nabla^\omega}$ as the formal dual of $d_{\nabla^\omega}$, then
$$\MRicHV = \sup_{B} \|\delta_{\nabla^\omega} \Omega\|_{\widecheck{\tensorg} 
\otimes \widecheck{\tensorv}}.$$ 
In particular, $\RicHV = 0$ if and only if $\delta_{\nabla^\omega}
\Omega = 0$, which is the definition of a Yang-Mills connection on
$\pi$.
\end{example}

\subsection{Invariant sub-Riemannian structures on Lie groups} \label{sec:LieGroup} 
Let $G$ be a Lie group with Lie algebra
$\frakg$. Let $G$ have dimension $n +\nu$. Choose a subspace $\frakh
\subseteq \frakg$ of dimension $n$ which generate the entire Lie
algebra, and give this subspace an inner product. Define a vector
bundle $\calH$ by left translation of $\frakh$. Use the inner product
on $\frakh$ to induce a left invariant metric $\tensorh$ on
$\calH$. This gives us a sub-Riemannian manifold $(G, \calH,
\tensorh)$ with a left-invariant structure, i.e. $G$ acts on the left
by isometries. This means that right invariant vector fields are
infinitesimal isometries, however, we cannot be sure that we have a
complement spanned by such vector fields. This is the case if and only
if there exist a subspace $\frakk$ of $\frakg$ such that
$\Ad(a)\frakk$ is a complement of $\frakh$ for any $a \in G$. Consider the special case when~$\frakk$ is a subalgebra of~$\frakg$
with corresponding subgroup~$K$ and $\frakg = \frakh \oplus \frakk$ as a vector space.
\begin{itemize}
\item[(a)] Define $\calV^l$ by left translation of $\frakk$. Then
  $\calV$ is a complement to $\calH$, but this is not in general
  spanned by infinitesimal isometries. If $K$ is closed, $\calH$ can
  be considered as an Ehresmann connection on $\pi: G \to G/K$, but it
  is not principal in general and we cannot necessarily consider the
  metric $\tensorh$ as lifted from~$G/K$.
\item[(b)] Define $\calV^r$ by right translation of $\frakk$. Then
  $\calV^r$ is spanned by infinitesimal isometries. It is a complement
  if and only if $\Ad(a) \frakk$ is a complement to $\frakh$ for every
  $a \in G.$ If the latter holds and $K$ is closed, $\calH$ can be
  considered as an Ehresmann connection on $\pi: G \to K\backslash G$.
\item[(c)] If $\frakk$ is an ideal (and $K$ a normal
  subgroup as a result) then $\calV^l = \calV^r$ is a complement spanned by
  infinitesimal isometries.
\end{itemize}

\begin{example}[Free step-2 nilpotent Lie groups] \label{ex:Nilpotent}
Let $\frakh$ be an inner product space of dimension $n$ and define $\frakk = \bigwedge^2 \frakh$ with the inner product induced by the product on $\frakh$. Define a Lie algebra $\frakg$ as the vector space
$\frakg = \frakh \oplus \frakk$ with brackets $[\newbullet, \newbullet]$ such that $\frakk$ is the center
and for any $A,B \in \frakh$,
$$[A,B] = A \wedge B \in \frakk.$$
Then $\frakg$ is a step~2 nilpotent Lie algebra of dimension ${n(n+1)}/{2}.$ Using the inner products on $\frakh$ and $\frakk$ and defining these two spaces as orthogonal, we get an inner product on $\frakg$.

Let $G$ be a simply connected Lie group with Lie algebra $\frakg$ and normal subgroup $K$ corresponding to $\frakk.$ Define $\calH$ and $\calV$ by left translation of respectively $\frakh$ and $\frakk$. Give $G$ a Riemannian metric by left translation of the inner product on $\frakg$. If we consider the inner product space $\frakh$ as a flat Riemannian manifold, then
$$\pi: G \to G/K \cong \frakh,$$
is a Riemannian submersion with $\ker \pi_* = \calV$ and with $\calH$ as an Ehresmann connection. Also $\widetilde \tensorg$ coincides with the privileged metric of Section~\ref{sec:PrivMetric}.

Since
$$\|\calR(v, \newbullet)\|_{\widetilde \tensorg^* \otimes \widetilde \tensorg}^2 = (n-1) \|\pr_{\calH} v\|_{\widetilde \tensorg}^2, \qquad v \in TM,$$
 we normalize the vertical part by defining
$$\tensorg = \pr_{\calH}^* \widetilde \tensorg|_{\calH} + \frac{1}{n-1} \pr_{\calV}^* \widetilde \tensorg|_{\calV}.$$
With respect to this normalized metric,
$$\McalR = 1, \qquad \mcalR = \frac{1}{\sqrt{n-1}}, \qquad \rnabla \tensorg = 0, \qquad \MRicHV = 0.$$
and $\rRicH = 0$ since $\frakh$ is flat. Defining $\srL$ as the sub-Laplacian of $\calV$
or equivalently the volume form of $\tensorg$, we obtain inequality
$$\sfGamma^{\tensorh^*+\ell \tensorv^*}_2(f) \geq \frac{1}{n} (\srL f)^2 
- \frac{1}{\ell} \sfGamma^{\tensorh^*}(f) + \frac{1}{2 (n-1)} \sfGamma^{\tensorv^*}(f).$$
\end{example}

\begin{remark}
  Let $\tensorg$ be a left invariant metric on the Lie group $G$, with
  $\tensorh$ and $\tensorv$ as the respective restrictions of
  $\tensorg$ to a left invariant subbundle $\calH$ and its orthogonal
  complement $\calV$. If $\calV$ is a metric-preserving
  complement of $(G, \calH, \tensorh)$, the conditions of
  Theorem~\ref{th:CD} hold, but we do not necessarily have $\rnabla
  \tensorg = 0$.
\end{remark}

\subsection{Sub-Riemannian manifolds with several metric-preserving
  complements} \label{sec:2Vcomp} The choice of metric preserving
complement may not be unique and give different results in general. We give two examples of this.
\begin{example}
Let $\frakg$ be a compact semi-simple Lie algebra of dimension $n$. The
term ``compact Lie algebra'' is here used to mean that the Killing
form $$(A,B) \mapsto \tr\,\ad(A)\ad(B)$$ is negative definite. We remark that when $\frakg$ is semi-simple,
then $[\frakg,\frakg] = \frakg$. Define and inner product
$$\langle A, B \rangle = - \frac{1}{4\rho} \tr\, \ad(A)\ad(B),$$
for some $\rho >0$. Note that if we use this inner product to induce a product
on $\End(\frakg) \cong \frakg^* \otimes \frakg$, then $\|\ad(A)\|^2 = 4\rho\|A\|^2$.

Let $G$ be a (compact) Lie group with with Lie algebra $\frakg$ and let $\widecheck{\tensorg}$ be
the Riemannian metric on $G$ obtained by left (or right) translation of $\langle \newbullet,\newbullet \rangle$.
From standard theory of bi-invariant metrics on Lie groups, it follows that $\Ric_{\widecheck{\tensorg}}(Z,Z) = \rho \|Z\|_{\widecheck{\tensorg}}^2$ pointwise for any vector field $Z$, so $G$ has Ricci lower bound $\rho >0$.

In what follows, we will always use the same symbol for an element in a Lie algebra and the corresponding left invariant vector field.
  \begin{enumerate}[(a)]
  \item Define
$$\frakh = \{ (A, 2 A) \in \frakg \oplus \frakg \, \colon \, A \in \frakg\}.$$
From our assumptions, we know that $\frakh + [\frakh, \frakh]= \frakg
\oplus \frakg$.  Define $\calH$ by left translation on $G \times
G$. Then $\calH$ is an Ehresmann connections of the following
submersions
$$\pi^j: G \times G \to G, \quad \pi^j(a_1, a_2) = a_j \quad j =1,2, (a_1, a_2) \in G.$$
Then the pullback of $\widecheck{\tensorg}$ by $\pi^1$ or of $\frac{1}{4} \widecheck{\tensorg}$ by $\pi^2$ gives us the same metric $\tensorh$ when restricted to $\calH$. We can write this as
$$\tensorh\left((A, 2 A), (A, 2 A) \right) = \langle A, A\rangle, \quad A \in \frakg.$$
The sub-Laplacian defined relative to either $\ker \pi_*^1$ or $\ker
\pi_*^2$ is
$$\srL = \sum_{i=1}^n (A_i, 2 A_i)^2$$
where $A_1, \dots, A_n$ is some orthonormal basis of $\frakg$. 
First consider, $\calV_1 = \ker \pi^1_* $ spanned by elements $(0,
A)$, $A \in \frakg$, with the metric $\tensorv_1$ given as
$$\|(0,A)\|_{\tensorv^1}^2 = \frac{1}{16\rho }  \langle A, A \rangle.$$
The constants we obtain are
\begin{align*}
&\calR( (A, 2 A), (B, 2 B) )= 2 (0, [A,B]), \quad \scrM_{\rnabla \tensorv_1^*} = 0,\\
&\McalR = 1, \quad \rRicH = \rho, \quad \MRicHV = \frac{3}{8} , \quad \mcalR = \frac{1}{2},
\end{align*}
giving us the inequality
\begin{equation} \label{firstCD} \sfGamma^{\tensorh^* + \ell \tensorv^*_1}_2(f) \geq \frac{1}{n} (\srL f)^2 + (\rho - c^{-1} - \ell^{-1} ) \sfGamma^{\tensorh^*}(f) +  \left( \frac{1}{8} - \frac{9c}{64} \right) \sfGamma^{\tensorv^*_1}(f),\end{equation}
for any $c > 0$. However, by choosing $\calV_2
= \ker \pi_*^2$, with metric
$$\|(A, 0)\|_{\tensorv^2}^2 = \frac{1}{ 4 \rho}  \langle A, A \rangle.$$
we obtain a better result
$$\calR( (A, 2 A), (B, 2 B)) = - ([A,B],0), \quad \scrM_{\rnabla \tensorv_2^*} =0,$$
$$\McalR =1, \quad \rRicH = 4 \rho, \quad \MRicHV = 0, \quad \mcalR = \frac{1}{\sqrt{2}},$$
so that
\begin{equation} \label{secondCD} \sfGamma^{\tensorh^* + \ell \tensorv^*_2}_2(f) \geq \frac{1}{n} (\srL f)^2 + (4 \rho - \ell^{-1} ) \sfGamma^{\tensorh^*}(f) + \frac{1}{4}  \sfGamma^{\tensorv^*_2}(f).\end{equation}
From Proposition~\ref{prop:SpectralGap} and Eq.~\eqref{secondCD}, we know that if $\lambda$ is any non-zero eigenvalue of $\srL$, then
$$\frac{ 4n }{5 n-1} \rho \leq -\lambda.$$
By contrast, we can not even obtain a spectral gap bound using inequality \eqref{firstCD} unless $\rho > {9}/{8}$, and even then, the result from using $\tensorv_2^*$ will give the better bound.

\item Consider $\real^n$ as the trivial Lie algebra. Let $I: \frakg \to \real^n$ be a bilinear map of vector
spaces. Define $\frakh$ as a subspace of $\frakg \times \real^n$ by
$(A, I(A))$, $A \in \frakg$. Consider $\frakg \times \real^n$ as the
Lie algebra of $G \times \real^n$, where $\real^n$ is considered
as a Lie group under~$+\,$. Define $\calH$
by left translation of $\frakh$. This is an Ehresmann connection
relative to both projections
$$\pi^1\colon G \times \real^n \to G, \qquad \pi^2\colon G \times \real^n \to \real^n.$$
Give $G$ the metric $\widecheck{\tensorg}$ and give $\real^n$ a flat
metric by the inner product $\langle I(\newbullet) , I(\newbullet)
\rangle$. Pulling back these metrics through respectively $\pi^1$ and
$\pi^2$, we obtain the same sub-Riemannian metric $\tensorh$ on
$\calH$ given by
$$\|(A, I(A))\|^2_{\tensorh} = \langle A, A \rangle,$$
even though the geometry of $G$ and $\real^n$ are very different. The
sub-Laplacians with respect to $\calV^1 = \ker\pi^1$ and $\calV^2 =
\ker\pi^2$ also coincide; it is given by
$$\srL = \sum_{i=1}^n (A, I(A))^2, \quad A_1, A_2, \dots, A_n \text{ an orthonormal basis of } \frakg.$$

We will leave out most of the calculations, and only state that if we
define $\tensorv_j$ on $\calV^j = \ker \pi^j_*$ by
$$\|(0, I(A))\|_{\tensorv_1}^2 = \frac{1}{4\rho} \langle A, A \rangle, \quad \|(A, 0)\|_{\tensorv_2}^2 = \frac{1}{4\rho} \langle A, A \rangle,$$
then these metrics are appropriately normalized and we get inequalities
\begin{align*}
  \sfGamma^{\tensorh^* + \ell \tensorv^*_1}_2(f) & \geq  \frac{1}{n} (\srL f)^2 + \left(\rho - \frac1c  - \frac1\ell\right) \sfGamma^{\tensorh^*}(f) +  \left(\frac{1}{4} - c \rho  \right)\sfGamma^{\tensorv_1^*}(f), \\
  \sfGamma^{\tensorh^* + \ell \tensorv^*_2}_2(f) & \geq \frac{1}{n}
  (\srL f)^2 - \left( \frac1c + \frac1\ell\right) \sfGamma^{\tensorh^*}(f) +
  \left(\frac{1}{4} - c \rho  \right) \sfGamma^{\tensorv_2^*}(f),
\end{align*}
that hold for any $c > 0$ and $\ell > 0$.
\end{enumerate}
\end{example}

\subsection{A non-integrable example}
As usual, we use the same symbol for an element of the Lie algebra and
the corresponding left invariant vector field. Consider the
complexification of $\mathfrak{su}(2)$ spanned over $\comp$ by
$$[A,B] = C, \quad [B,C] = A, \quad [C,A] = B.$$
Consider $\mathfrak{su}(2)^{\comp}$ as the Lie algebra of
$\mathrm{SL}(2,\comp)$ and on that Lie group, define (real) subbundles
of the tangent bundle
$$\calH = \spn \{i A, iB, iC, C \}, \quad \calV = \spn \{ A, B \}.$$
Define a Riemannian metric $\tensorg$ such that $iA, iB, iC, C,
\sqrt{2} A$ and $\sqrt{2} B$ forms an orthonormal basis. Define
$\tensorg|_{\calH} = \tensorh$ and $\tensorg|_{\calV} = \tensorv$ and
let $\roughL$ be the sub-Laplacian of the sub-Riemannian manifold
$(\calH, \tensorh)$ with respect to the complement $\calV$. It is
simple to verify that
$$\roughL = (iA)^2 + (iB)^2 + (iC)^2 +C^2.$$
We caution the reader that $(iA)^2$ stands for the left invariant vector field 
corresponding the element $iA$ applied twice; it is in no way equal to $-A^2.$

Let $\calR$ and $\overline{\calR}$ be respectively the curvature and the co-curvature of
$\calH$. Note that
\begin{align*}
&\rnabla \tensorh^* = 0, \quad \rnabla \tensorv^* = 0,\\
&\tr \overline{\calR}(v, \calR(v, \newbullet)) = 0 \quad \text{for any } v \in TM,\\
&\RicH(Y, Y) = -\frac{5}{2} \tensorg(iA, Y)^2 - \frac{5}{2} \tensorg(iB, Y)^2 -2 \tensorg(iC, Y)^2 + \frac{1}{2} \tensorg(C, Y)^2,\\
&\RicHV = 0, \quad \McalR = 1, \quad \mcalR =1.
\end{align*}
It follows that $\roughL$ is also the sub-Laplacian of the volume form
of $\tensorg$ with curvature-dimension inequality
$$\sfGamma^{\tensorh^* + \lambda \tensorv^*}_2(f) \geq \frac{1}{4} (\roughL f)^2 - \left(\frac{5}{2} + \frac{1}{\lambda} \right) \sfGamma^{\tensorh^*}(f) + \frac{1}{2} \sfGamma^{\tensorv^*}(f).$$

\section{Summary of Part~II} \label{sec:PartII} We include a section
here to illustrate what further results can be obtained from our
curvature-dimension inequality \eqref{CDstar} of Theorem~\ref{th:CD}.

Let $L$ be a second order operator. Let $X(x)$ be an
$L$-diffusion with $X(x) = x$ and maximal lifetime
$\tau(x)$. For bounded functions $f$, define $P_t f$ by $$P_tf(x) =
\expect[f(X_t(x)) 1_{t \leq \tau}].$$ Assume that $L$ satisfies
an inequality \eqref{CDstar} with respect to some $\tensorv^* \in
\Gamma(\Sym^2 TM)$. Let $C^\infty_b(M)$ denote the space of all smooth, bounded functions.
We will introduce two important conditions.
\begin{itemize}
\item[({\sf A})] $P_t 1 = 1$ and for any $f \in C_b^\infty(M)$ with $\sfGamma^{\tensorh^* + \tensorv^*}(f) \in C^\infty_b(M)$ and for every $T > 0$, we have
$$\sup_{t \in [0,T]} \| \sfGamma^{\tensorh^* + \tensorv^*}(P_t f) \|_{L^\infty}  < \infty.$$
\item[({\sf B})] For any $f \in C^\infty(M)$, we have
  $\sfGamma^{\tensorh^*}(f,\sfGamma^{\tensorv^*}(f)) =
  \sfGamma^{\tensorv^*}(f, \sfGamma^{\tensorh^*}(f)).$
\end{itemize}
We have the following concrete classes of sub-Riemannian manifolds
satisfying these conditions.
\begin{theorem}[\rm{Part~II, Proposition~3.3, Theorem~3.4}] \label{th:Cond} 
  \begin{enumerate}[\rm(a)]
  \item Assume that $\pi\colon M \to B$ is a fiber-bundle with compact
    fibers over a Riemannian manifold $(B, \widecheck{\tensorg})$. Let
    $\calH$ be an Ehresmann connection on $\pi$ and define $\tensorh=
    \pi^* \widecheck{\tensorg}|_{\calH}$. Assume that the metric
    $\metricd_{cc}$ of $(M, \calH, \tensorh)$ is complete and that the
    sub-Laplacian $\roughL$ of $\calV = \ker \pi_*$ satisfies
    \eqref{CDstar}. Then condition \emph{({\sf A})} holds.
  \item Let $M$ be a sub-Riemannian manifolds with an integrable,
    metric preserving complement~$\calV$. Assume that there exist a
    choice of $\tensorv$ on $\calV$ such that $\tensorg =
    \pr_{\calH}^* \tensorh + \pr_{\calV}^* \tensorv$ is a complete
    Riemannian metric and such that $\rnabla \tensorv^* = 0$. Then the
    sub-Laplacian $\srL$ of $\calV$ and the volume form of $\vol$ of
    $\tensorg$ coincide. Assume that $\srL$ satisfies the assumptions of
    Theorem~\textup{\ref{th:CD}}. Finally, assume that $\mcalR
    >0$. Then conditions \emph{({\sf A})} and \emph{({\sf B})} hold.
  \end{enumerate}
\end{theorem}
The result of Theorem~\ref{th:Cond}~(b) is also valid for non-integrable complements
of the type described in Section~\ref{sec:NotIntegrable}. For the special case when $\tensorh^*$ is a
complete Riemannian metric, $\tensorv^* = 0$ and $L$ is the
Laplacian, ({\sf A}) always holds when the curvature-dimension
inequality holds. For this reason, we expect that the condition ({\sf
  A}) will hold in more cases than the ones listed here.

Combined with our generalized curvature-dimension inequality, we have
the following identities.
\begin{theorem}
  \emph{(Part~II, Proposition~3.8)} Assume that $L$ satisfies
  \eqref{CDstar} with respect to $\tensorv^*$ and that \emph{({\sf
      A})} holds. Assume also that $L$ is symmetric with respect to
  the volume form $\vol$, i.e. $\int_M fLg \dvol = \int_M gLf \dvol$ for any
  $f,g \in C_c^\infty(M)$. Finally, assume that $\tensorg^* =
  \tensorh^* + \tensorv^*$ is the co-metric of a complete Riemannian
  metric.
  \begin{enumerate}[\rm(a)]
  \item If $\rho_1 \geq \rho_{2,1}$ and $\rho_{2,0}>-1$, then for any compactly supported $f \in C_c^\infty(M)$,
$$\| \sfGamma^{\tensorh^*}( P_t f) \|_{L^1} \leq e^{- \alpha t} \| \sfGamma^{\tensorh^*}(f)\|_{L^1}, \quad \alpha := \frac{\rho_{2,0} \rho_1 + \rho_{2,1}}{\rho_{2,0} +1}.$$
Furthermore, if $\alpha > 0$ then $\vol(M) < \infty$.
\item Assume that the conditions in \emph{(a)} hold with $\alpha
  >0$. Then for any $f \in C_c^\infty(M)$,
$$\| f- f_M \|_{L^2}^2 \leq \frac{1}{\alpha} \int_M \sfGamma^{\tensorh^*}(f) \dvol$$
where $f_M = {\vol(M)}^{-1} \int_M f \dvol$. As a consequence if
$\lambda$ is a non-zero eigenvalue of~$L$, then $\alpha \leq - \lambda.$
\end{enumerate}
\end{theorem}

There are also other inequalities which do not require that $L$ is
symmetric with respect to a volume form, see e.g. Part~II, Proposition~3.6.

For the case when both ({\sf A}) and ({\sf B}) hold and $L$
satisfies \eqref{CD} of Section~\ref{sec:NablaMetric0} with $\rho_2 >0$,
we can use the results from \cite{BaGa12,BaBo12,BBG14}. In particular,
with some extra computation, we can conclude the following.

\begin{corollary}[\rm Part~II, Proposition~5.1, Proposition~5.3]\
Let $\srL$ is as in \emph{Theorem \ref{th:Cond}~(b)}. Let
$$\kappa = \tfrac{1}{2} \mcalR \rRicH - \MRicHV^2 .$$
\begin{enumerate}[\rm(a)]
\item Assume that $\kappa > 0$. Then $M$ is compact, and for any $f\in C^\infty(M)$,
$$\| f- f_M \|_{L^2}^2 \leq \frac{1}{\alpha} \int_M \sfGamma^{\tensorh^*}(f) \dvol$$
where $f_M = {\vol(M)}^{-1} \int_M f \dvol$ and
$$\alpha  := \left( \frac{2 \kappa}{2 \MRicHV+\mcalR \sqrt{2 \rRicH + 2\kappa }} \right)^2.$$
\item Assume that $\kappa \geq 0$ and define
$$N = \frac{n}{4} \frac{ \left(\sqrt{2 \rRicH + \kappa} + \sqrt{\rRicH+ \kappa }\right)^2}{\kappa}, \qquad D = \frac{\sqrt{(\kappa+ \rRicH)(\kappa  + 2\rRicH)}}{\kappa}.$$ 
Let $p_t(x, y)$ be the heat
  kernel of $\frac{1}{2} \srL$ with respect to $\vol$. Then
$$p_t(x,x) \leq \frac{1}{t^{N/2}} p_1(x,x)$$
where $x \in M$ and $0 \leq t \leq 1$.
Furthermore, for any $0 < t_0 < t_1$, any non-negative smooth bounded function $f \in C^\infty_b(M)$ and points $x,y \in M$,
$$P_{t_0} f(x) \leq (P_{t_1}f)(y) \left(\frac{t_1}{t_0}\right)^{N/2} \!\!\!\exp\left( D \frac{ \metricd(x,y)^2}{2 (t_1-t_0)}  \right).$$
If $\kappa =0$, we interpret the quotient ${\kappa}/{\rRicH}$ as
$\mcalR^2/2$.
\end{enumerate}
\end{corollary}

\bibliographystyle{abbrv}      
\bibliography{Bibliography.bib}

\end{document}